\documentclass{amsart}
\usepackage{amsmath}
\usepackage{array}
\usepackage{stackengine}

\usepackage{amsfonts}
\usepackage{url} 
\usepackage{xcolor} 
\usepackage{fancyhdr}
\usepackage{multirow}
\usepackage{graphicx}
\usepackage{relsize}
\usepackage{amssymb}
\usepackage[normalem]{ulem}
\usepackage{xcolor}
\definecolor{tealedit}{RGB}{0,128,128}

\usepackage{hyperref}
\usepackage{mathtools}
\usepackage{indentfirst}
\usepackage{comment}
\usepackage{enumerate}
\usepackage{float}

\usepackage{tikz}
\usetikzlibrary{arrows.meta}

\newtheorem{theorem}{Theorem}[section]
\newtheorem*{question*}{Question}

\newtheorem{example}[theorem]{Example}
\newtheorem{lemma}[theorem]{Lemma}
\newtheorem{corollary}[theorem]{Corollary}
\newtheorem{proposition}[theorem]{Proposition}
\newtheorem{remark}[theorem]{Remark}

\newtheorem{definition}[theorem]{Definition}

\hypersetup{colorlinks=true,linkcolor=blue, linktocpage}
\hypersetup{colorlinks,
citecolor=blue,
allcolors=blue
}
\DeclareMathOperator{\dist}{\rho}
\DeclareMathOperator{\diam}{diam}
\DeclareMathOperator{\Fix}{Fix}
\DeclareMathOperator{\Int}{int}
\DeclareMathOperator{\Rec}{Rec}
\DeclareMathOperator{\Per}{Per}
\DeclareMathOperator{\End}{End}

\DeclareMathOperator{\htop}{h_{top}}
\DeclareMathOperator{\C}{C}
\DeclareMathOperator{\lcm}{lcm}
\DeclareMathOperator{\N}{N}

\newcommand{\eps}{\varepsilon}

\title[$\alpha$-limit sets in hyperspace]
{Special $\alpha$-limit sets in the hyperspace of continua: closedness and entropy for interval maps}
\author[D. Jeli\'c]{Domagoj Jeli\'c}
\address[D. Jeli\'c]
{Faculty of Science, University of Split, Ru\dj era Bo\v{s}kovi\'ca 33, 21000 Split, Croatia
}
\email{djelic@pmfst.hr}
\author[P. Oprocha]{Piotr Oprocha}
\address[P. Oprocha]{
	Centre of Excellence IT4Innovations - Institute for Research and Applications of Fuzzy Modeling, University of Ostrava, 30. dubna 22, 701 03 Ostrava 1, Czech Republic.
}
\email{piotr.oprocha@osu.cz}
\subjclass{37E05, 37B40, 54B20.}
\keywords{interval map, special $\alpha$-limit set, $\alpha$-limit set, backward dynamics, hyperspace of continua, topological entropy, horseshoe.}
\date{\today}
\begin{document}

\begin{abstract}
This paper investigates the backward dynamics of hyperspace systems induced by
continuous interval maps.
Focusing on the hyperspace of continua, we provide a structural
characterization of the $\alpha$-limit sets of backward branches for interval subcontinua.
A central result of this work is the proof that the special $\alpha$-limit set of any
nondegenerate subinterval is always closed. This reveals a striking topological contrast
with classical single-point dynamics, where special $\alpha$-limit sets
need not be closed.
Finally, we prove that if a special $\alpha$-limit set contains two nondegenerate periodic continua with disjoint orbits, then the base map has a horseshoe and, consequently, positive topological entropy.
\end{abstract}

\maketitle

\section{Introduction}

Although dynamical systems theory traditionally focuses on the asymptotic behavior of individual trajectories, studying the evolution of larger sets often reveals profound features of the system that remain invisible at the level of single points. This macroscopic perspective naturally leads to the investigation of hyperspaces. For a given compact metric space $X$ and a continuous self-map $f\colon X \to X$, one can naturally define the induced mappings on the hyperspace of all nonempty compact subsets $2^X$, as well as on the hyperspace of continua $\mathcal{C}(X)$. Equipped with the Hausdorff metric, both $2^X$ and $\mathcal{C}(X)$ become compact metric spaces, and the induced maps naturally constitute dynamical systems in their own right.

Understanding to what extent the dynamical properties of the base map $f$ are inherited by, or reflected in, the induced hyperspace systems has been a vibrant area of research since the classical work of Bauer and Sigmund in 1975~\cite{Bauer}. Over the decades, a vast literature has emerged. 
For instance, the transfer of transitivity, mixing, and weak mixing to hyperspaces has been thoroughly investigated (see, e.g., \cite{Dendrite1, Banks, KOCSF}). More recently, attention has shifted to shadowing properties~\cite{Arbieto, Carvalho, Fernandez} and various measures of dynamical complexity, such as polynomial entropy and mean dimension~\cite{Djoric, Jelic, Zhu}. 

A particularly interesting aspect of this complexity is the behavior of topological entropy. It is well known that if a base map has positive topological entropy, the entropy of the induced system on the full hyperspace $2^X$ typically blows up to infinity~\cite{KOCSF}. In stark contrast, the hyperspace of continua $\mathcal{C}(X)$ exhibits much more rigid behavior: for spaces, such as intervals and graphs, the topological entropy on $\mathcal{C}(X)$ coincides exactly with the entropy of the base system (see, e.g., \cite{MatviiInterval}),
however even in the one-dimensional case, such as dendrites, it can explode to infinity \cite{Dendrite2}.

Despite the comprehensive understanding of forward dynamics and complexity in these spaces, the corresponding backward dynamics remains significantly less explored. The concept of an $\alpha$-limit set was originally introduced as a dual notion to the $\omega$-limit set, aiming to describe the asymptotic source of a trajectory. For noninvertible continuous maps, however, a single point may have multiple preimages, leading to an intricate branching of backward orbits. This inherent difficulty in tracking the past of a system naturally gave rise to several non-equivalent definitions of negative limit sets. 

The earliest rigorous approach, introduced by Coven and Nitecki in 
\cite{Coven}, defined the $\alpha$-limit set based on the accumulation of full preimages of a given point. A fundamental refinement was later introduced by Hero~\cite{Hero}, who shifted the focus to individual backward branches. He defined the \emph{special} $\alpha$-limit set, denoted as $s\alpha(x)$, as the set of limit points of all possible backward branches originating from $x$. Over the years, the subtle topological and dynamical differences between these various definitions have been thoroughly systematized (see, e.g., \cite{Balibrea_def, CholewaOprocha2021}). 

Recently, the study of special $\alpha$-limit sets has gained immense momentum, driven by deep topological and dynamical questions. Kolyada, Misiurewicz, and Snoha~\cite{KolyadaMisiurewiczSnoha2020} significantly advanced the theory by providing a comprehensive investigation of the fundamental properties of these sets. Their work inspired extensive research by a broader mathematical community. For instance, Jackson, Mance, and Roth~\cite{JacksonManceRoth2022} explored the complex descriptive set-theoretic nature of special $\alpha$-limit sets (such as non-Borel properties in dimension two), while Hant\'akov\'a and Roth~\cite{HantakovaRoth} investigated their role in one-dimensional spaces. Some further characterizations of $\alpha$-limit sets in one-dimensional dynamics were obtained in~\cite{Forys}.

A fundamental topological question at the heart of these recent developments is whether special $\alpha$-limit sets are always closed. In general, the topological complexity of these sets can be surprisingly wild; for instance, Jackson, Mance, and Roth~\cite{JacksonManceRoth2022} demonstrated that for a continuous map on the unit square, $s\alpha(x)$ can be an analytic but non-Borel set. For continuous interval maps, however, the structure was expected to be much more regular, leading Kolyada, Misiurewicz, and Snoha~\cite{KolyadaMisiurewiczSnoha2020} to formally raise the question of whether special $\alpha$-limit sets in this one-dimensional setting are always closed. This was definitively answered by Hant\'akov\'a and Roth~\cite{HantakovaRoth}, who proved that $s\alpha(x)$ is not necessarily closed, even for interval maps. They provided a complete characterization of this phenomenon, demonstrating that $s\alpha(x)$ fails to be closed if and only if $x$ belongs to a maximal solenoidal set containing a nonrecurrent point from the Birkhoff center.

While the topological nature of special $\alpha$-limit sets for individual points is now well understood, their behavior in the realm of hyperspaces remains completely uncharted. The present paper takes this next natural step, extending the intricate investigation of backward dynamics to the hyperspace of continua. It serves as a natural continuation of previous works~\cite{Jelic, Jelic2}, which studied entropy, recurrent points, and $\omega$-limit sets of induced systems on one-dimensional continua. Let $f\colon I\to I$ be an interval map and $A\in\mathcal{C}(I)$. In this work, we address two central questions:
\begin{enumerate}
    \item For a given backward branch $\{A_n\}_{n\leq 0}$ of $A$, what is the exact structure of $\alpha(\{A_n\}_{n\leq 0})$? 
    \item  Is $s\alpha(A)$ always closed, and what can be said about its topological structure in general?
\end{enumerate}

For the base map, the first question was almost completely answered by Fory\'s-Krawiec, Hant\'akov\'a, and Oprocha~\cite{Forys}. In this paper, we expand this theory to the induced system, providing a complete structural description of the $\alpha$-limit sets of backward branches in $\mathcal{C}(I)$.

\begin{theorem}\label{thm:alpha_char}
Let $I$ be a compact interval and let $\{A_n\}_{n\leq 0}$ be a backward branch in $\mathcal{C}(I)$. 
Then one of the following cases holds:
\begin{enumerate}[(a)]
    \item $\alpha(\{A_n\}_{n\leq 0})$ is a single periodic orbit;
    \item $\alpha(\{A_n\}_{n\leq 0})=\{\{x\}\colon x\in L\}$, where $L$ is an $\alpha$-limit set of the base map;
    \item There is a maximal basic set $\omega$ in $(I,f)$ and a point $x\in A$ with backward branch $\{x_n\}_{n\leq 0}$, $x_n\in A_n$ for all $n\leq 0$, with infinite $\alpha(\{x_n\}_{n\leq 0})=L\subset \omega$, such that each element of $\alpha(\{A_n\}_{n\leq 0})$ intersects $L$ and is either a singleton subset of $L$ or a wandering interval with endpoints from $L$. Moreover, $L\subset \bigcup \alpha(\{A_n\}_{n\leq 0})$. 
\end{enumerate}
\end{theorem}

Informally, such an $\alpha$-limit set is governed by the $\alpha$-limit set $L$ of a single point, possibly enriched by wandering intervals with endpoints in $L$.

More strikingly, we discover a profound contrast between point dynamics and hyperspace dynamics regarding the closedness problem. While $s\alpha(x)$ for single points can fail to be closed, we prove that if a subinterval $A$ is nondegenerate, its special $\alpha$-limit set $s\alpha(A)$ is always closed.

\begin{theorem}\label{thm:special_closed}
Let $f\colon I\to I$ be an interval map and $A\in \mathcal{C}(I)$ nondegenerate. Then $s\alpha(A)$ is closed.
\end{theorem}

Finally, we connect the backward dynamics in the hyperspace directly to the topological complexity of the base map. We establish that the presence of disjoint periodic orbits with nondegenerate elements within $s\alpha(A)$ guarantees the existence of a horseshoe, thereby ensuring positive topological entropy of the base system $(I,f)$.

\begin{theorem}\label{thm:pos_ent}
Let $I$ be a compact interval, $A\in\mathcal{C}(I)$ and let there be nondegenerate $B,C\in\Per(\tilde{f})\cap s\alpha(A)$ such that $\mathcal{O}_{f}(B)\cap \mathcal{O}_{f}(C)=\emptyset$. Then $h_{top}(f)>0$.
\end{theorem}

Therefore, by Theorems~\ref{thm:alpha_char} and~\ref{thm:pos_ent}, whenever we have an interval map $f\colon I\to I$ of zero topological entropy and $A\in\mathcal{C}(I)$, the orbits of any two nondegenerate periodic $B,C\in s\alpha(A)$ must intersect.

\medskip
\noindent \textbf{Organization of the paper.}
In Section~\ref{sec:2}, we recall basic definitions and preliminary results concerning interval dynamics, maximal $\omega$-limit sets, and hyperspaces. Section~\ref{sec:3} provides preparatory lemmas regarding fixed points and the fundamental behavior of the induced maps. In Section~\ref{sec:4}, we analyze the case of periodic nondegenerate intervals, proving that nondegenerate, asymptotically periodic elements of $\alpha$-limit sets in the induced system are strictly periodic. Section~\ref{sec:5} is devoted to the study of wandering intervals, where we construct an illustrative example using a tent map and formally describe the structure of such intervals. In Section~\ref{sec:6}, we synthesize the preceding results to completely characterize the $\alpha$-limit sets of backward branches in $\mathcal{C}(I)$, thereby proving Theorem~\ref{thm:alpha_char}. Finally, Section~\ref{sec:7} focuses on special $\alpha$-limit sets; we present several motivating examples and prove our main results concerning the closedness of $s\alpha(A)$ (Theorem~\ref{thm:special_closed}) and the sufficient condition for positive topological entropy (Theorem~\ref{thm:pos_ent}).

\section{Preliminaries}\label{sec:2}
\subsection{Dynamical systems}
	A (discrete, topological) \emph{dynamical system} is a pair $(X,f)$ where {$(X,\dist)$} is a compact metric space and $f\colon X\to X$ is a continuous map.
	We will often identify the dynamical system $(X,f)$ with the map $f$.
	Any continuous map $f\colon I\to I$, where $I$ is a compact interval, will simply be called an \emph{interval map}.
	For a point $x\in X$, we define its \emph{orbit} as the set $\mathcal{O}_f(x)=\{f^n(x)\colon n\geq 0\}$ and its \emph{trajectory} as the sequence $\left(f^n(x)\right)_{n\geq 0}$.
{Similarly, the \emph{orbit of a set} $A\subset X$ is the set $\mathcal{O}_f(A)=\bigcup_{n=0}^{\infty}f^n(A).$
	When the map $f$ is clear from the context, we simply write $\mathcal{O}(x)$.}
	A point $x\in X$ is said to be \emph{periodic} if there is some $p>0$ such that $f^p(x)=x.$
	The smallest such $p$ is called the \emph{period} of $x$.
	We denote the set of all periodic points of $f$ by $\Per(f)$.
	A point $x\in X$ is said to be \emph{eventually periodic} if there is some $n\geq 0$ such that $f^n(x)\in\Per(f).$ 
	Equivalently, a point is eventually periodic if and only if it has finite orbit.

	If $(X,f)$ and $(Y,g)$ are two dynamical systems, a \emph{semi-conjugacy} between $f$ and $g$ is any surjective continuous map $\phi\colon X\to Y$ such that $\phi\circ f=g\circ \phi$.
	If in addition $\phi$ is bijective, and hence a homeomorphism, then it is a \emph{conjugacy} between $f$ and $g$ and we say that these two dynamical systems are \emph{conjugate}.

	A set $A\subset X$ is said to be {\emph{$f$-invariant}} if $f(A)\subset A$ and {\emph{strongly $f$-invariant}} if $f(A)=A.$
	A closed invariant set $M\subset X$ without a closed and invariant proper subset is called \emph{minimal}.
	It is well known that a closed invariant set $M\subset X$ is minimal if and only if the orbit of each point $x\in M$ is dense in $M$.	
	If $X$ is minimal, we say that dynamical system $(X,f)$ is minimal or, simply,  
	that $f$ is \emph{minimal map}.
	
	The map $f\colon X\to X$ is \emph{transitive} if for every pair of nonempty open subsets $U, V\subset X$ there is some $n>0$ such that $f^n(U)\cap V\neq\emptyset$.
	If $X$ does not contain isolated points, then $f$ is transitive if and only if there is $x\in X$ whose orbit is dense in $X$.	
	A map $f\colon X\to X$ is \emph{mixing} if for every pair of nonempty open subsets $U, V\subset X$ there is an $N\geq 0$ such that $f^n(U)\cap V\neq \emptyset$ for all $n\geq N$.	
	It is clear that each mixing map is transitive but not vice-versa.
	For a point $x\in X$, the limit set of its trajectory is called \emph{$\omega$-limit set of $x$} and denoted by $\omega_f(x)$.
	In other words, $y\in\omega_f(x)$ if and only if there is a strictly increasing sequence of positive integers $(n_i)_{i\geq 0}$ such that $\left(f^{n_i}(x)\right)\to y$ as $i\to\infty$.
	A point $x\in X$ is \emph{non-wandering} if for every neighborhood $U$ of $x$ and every $N>0$ there is some $n>N$ such that $f^n(U)\cap U\neq\emptyset.$	Otherwise, we call it a \emph{wandering} point.
	A point $x\in X$ is \emph{recurrent} if $x\in\omega_f(x)$. 
	The set of all recurrent points is denoted by $\Rec (f)$ and we denote $\omega(f)=\bigcup_{x\in X}\omega_f(x).$
	We also define $\C(f)=\overline{\Rec(f)}$ and call this set the \emph{(Birkhoff) center} of $(X,f)$.	
    Finally, a point $x\in X$ is said to be \emph{asymptotically periodic} if $\omega_f(x)$ is a periodic orbit.
	
	A \emph{backward branch} of a point $x\in X$ is any sequence $\{x_i\}_{i\leq 0}$ in $X$ such that $x_0=x$ and $f(x_i)=x_{i+1}$ for each $i<0$. 
	A point $y$ belongs to the \emph{$\alpha$-limit set of a backward branch $\{x_i\}_{i\leq 0}$}, denoted by $\alpha_f\left(\{x_i\}_{i\leq 0}\right)$, if and only if there is a strictly decreasing sequence of negative integers $\{n_i\}_{i\geq 0}$ such that $x_{n_i}\to y$ as $i\to \infty$.
	It is easy to
see that both $\omega$-limit sets and $\alpha$-limit sets of backward branches are closed strongly
invariant sets.
Finally, a point $y$ belongs to the \emph{special $\alpha$-limit set of a point $x$}, also written as \emph{$s\alpha$-limit set} and denoted by $s\alpha_f(x)$, if and only if some backward branch of
$x$ has a subsequence $\{x_{n_i}\}_{i=0}^\infty$ such that $x_{n_i}\to y$.
If $f$ is clear from the context, we will simply write $\alpha\left(\{x_i\}_{i\leq 0}\right)$ and $s\alpha(x)$.
	
\begin{definition}
	Let $f\colon X\to X$ and $g\colon Y\to Y$ be two continuous maps of compact metric spaces $X,Y$ and {$M\subset X$} be a closed invariant set.
	A semi-conjugation $\phi\colon X\to Y$ is an \emph{almost conjugacy between $f\vert_M$ and $g$} if
	\begin{enumerate}[(1)]
		\item\label{almost1} $\phi(M)=Y$,
		\item $\forall y\in Y$, $\phi^{-1}(y)$ is connected,
		\item $\forall y\in Y$, $\phi^{-1}(y)\cap M=\partial\phi^{-1}(y)$, where $\partial A$ denotes the boundary of $A$ in $X$,
		\item\label{alomst4} $\exists N\geq 1$ such that $\forall y\in Y$, $\phi^{-1}(y)\cap M$ has at most $N$ elements (and, by \eqref{almost1}, at least one element).
	\end{enumerate}
\end{definition}

When $X$ and $Y$ are arcs, the last condition is a consequence of the previous ones.

A typical example of a dynamical system semi-conjugated to an irrational rotation is the Denjoy map (see~\cite[Example~14.9]{Devaney}; we comment on it more later when introducing circumferential sets). Another simple example is the truncated tent map $f$ on $[0,1]$ given for some $s>2$ by 
$$
f(x)=\begin{cases}
sx &, x\in [0,1/s]\\
s(1-x) &, x\in [1-1/s,1]\\
1 &, \text{otherwise}
\end{cases}
$$ 

Let $\phi\colon [0,1]\to [0,1]$ be a map that collapses the interval $J=[1/s,1-1/s]$, as well as all intervals in its preimages $f^{-n}(J)$, to single points. After a proper rescaling of the domain, $f$ transforms into the standard tent map  
$$
g(x)=\begin{cases}
2x &, x\in [0,1/2]\\
2(1-x) &, x\in [1/2,1]
\end{cases}.
$$ 
Here, $\phi$ acts as an almost conjugacy between $f\vert_M$ and $g$, where $M=[0,1]\setminus \bigcup_{n\geq 0}f^{-n}((1/s,1-1/s))$.

\begin{remark}
It is clear that if $f\vert_M$ and $g$ are almost conjugate by a map $\phi$, then $\phi|_M$ is a semi-conjugacy, but it is not an almost conjugacy except in the case that it is a
conjugacy itself.
\end{remark}

\subsection{Maximal \texorpdfstring{$\omega$}{omega}-limit sets}

We will often use an important property of interval maps, namely that every $\omega$-limit set of an interval map is contained in a maximal one.
By the results of Blokh~\cite{B_spectral}, there are exactly three types of these maximal $\omega$-limit sets: periodic orbits, basic sets and solenoidal $\omega$-limit sets. We briefly recall this result below.

The set $\mathcal{O}(K)=\bigcup_{i=0}^kf^i(K)$ is called a \emph{cycle of intervals of period $k$} if $K$ is a subinterval of $I$ such that $K,f(K),...,f^{k-1}(K)$ are pairwise disjoint and $f^k(K)=K.$  
A \emph{generating sequence} or a \emph{sequence generating a solenoidal set} is any nested sequence of cycles of intervals $K_1\supset K_2\supset\cdots$ for $f$ with periods tending to infinity. By definition $Q=\bigcap_n K_n$ is closed and strongly invariant, i.e. $f(Q)=Q$. Any closed and strongly invariant subset of $Q$ is called a \emph{solenoidal set}. The characterization of Blokh shows that $Q$ contains a perfect minimal set $Q_{min}=Q\cap\overline{\Per (f)}$ such that $Q_{min}=\omega(x)$, for all $x\in Q$, and a maximal $\omega$-limit set (with respect to inclusion) $Q_{max}$ such that $Q_{max}=Q\cap \omega(f)$ which we will refer to as \emph{maximal solenoid}. 

A \emph{basic set} is an $\omega$-limit set which is infinite, maximal among $\omega$-limit sets, and contains some periodic point. If $B$ is a basic set then with respect to inclusion there is a minimal cycle of intervals $K$ that contains it, and $B$ may be characterized as the set of those points $x \in K$ such that $\mathcal{O}(\overline{U}) = K$ for every relative neighborhood $U$ of $x$ in $K$, see~\cite{B_spectral}. Conversely, if $M$ is a cycle of intervals for $f$, then we will write
\[
B(M,f) = \{ x \in M : \mathcal{O}(\overline{U}) = M \text{ for any relative neighborhood } U \text{ of } x \text{ in } M \},
\]
and if this set is infinite, then it is a basic set.
When dealing with basic sets we will often use a model map, described in the following lemma. The reader is referred to~\cite{B_spectral} and~\cite{Forys} for more details.
\begin{lemma}\label{lem:modelBS}
	Let $f\colon I\rightarrow I$ be an interval map and $X\subseteq I$ be a cycle of intervals. Suppose that $B(X,f)$ is a basic set. 
	Then there is a transitive map $g\colon Y\rightarrow Y$, where $Y$ is a cycle of intervals $Y_0,\ldots, Y_{n-1}$ with possibly non-empty intersection in the endpoints, and $\phi\colon X\rightarrow Y$ which almost conjugates $f|_{B(X,f)}$ and $g$.
	 Moreover, $g^n|Y_i$ is mixing, for $i=0,\ldots,n-1$.
	  The period $n$ of $Y$ is either equal to or double the period of $X$ and $Y_i\cap Y_j=\End(Y_i)\cap \End(Y_j)\neq \emptyset$ if and only if $i\neq j$ and $i$ and $j$ are congruent modulo the period of $X$.	
\end{lemma}
\begin{remark}\label{rem:model_BS}
For an arbitrary basic set $\omega=B(X,f)$, under the notation from Lemma~\ref{lem:modelBS}, $X_{(i)}=\phi^{-1}(Y_i)$ defines a minimal (in the sense of inclusion) intervals $\{X_{(i)},i=1,2,\ldots,n\}$ with possibly nonempty intersections such that $\omega\subset \bigcup_i X_{(i)}$, 
the intervals $X_{(i)}$ are cyclically permuted by $f$,
and $m$ is such that $f^m\vert_{X_{(i)}}$ is almost conjugated to a mixing map.
Furthermore, for every $A\in\mathcal{C}(I)$ and every $i=1,2,\ldots,n$ such that {$A\subset X_{(i)}$} and $\Int A\cap \omega\neq\emptyset$, $\omega_{\tilde{f}}(A)={\{X_{(i)},\ i=1,2,\ldots,n\}}$.
\end{remark}

Let us also recall the following result which connects topological entropy to the maximal $\omega$-limit sets of interval maps (see e.g.~\cite{Schweizer}).
\begin{theorem}
    Let $f\colon I\to I$ be an interval map. Then the $\htop(f)>0$ if and only if $f$ has a basic set.
\end{theorem}

The following two theorems provide almost complete description of $\alpha$-limit sets of backward branches for the graph maps, showing their close connection with the $\omega$-limit sets.

\begin{theorem}\label{thm:maximal}\cite[Theorem~4.6]{Forys}
Let $f$ be a continuous map acting on a graph $G$. If a set $L$ is an $\alpha$-limit set of a backward branch $\{x_j\}_{j\leq 0}$ then $L$ is contained in a maximal $\omega$-limit set.
\end{theorem}

\begin{theorem}\label{thm_alpha_zero_ent}\cite{Forys}
Let $f$ be a continuous map acting on a graph $G$ with $\htop(f) = 0$. Then a set $L$ is an $\alpha$-limit set of a backward branch $\{x_j\}_{j\leq 0}$ if and only if $L$ is a minimal set.
\end{theorem}

{
\begin{remark}\label{rem:alpha}
By Theorem~\ref{thm_alpha_zero_ent}, for interval maps having zero entropy, $\alpha$-limit sets are exactly periodic orbits and minimal solenoidal sets.

On the other hand, in the case of interval maps having basic sets, the situation is not completely clear.
	However, by the results of~\cite{Forys}, for any graph map, whether or not it had zero entropy, the $\alpha$-limit set of any backward branch is the $\omega$-limit set of some point.
 
\end{remark}
}
\subsection{Hyperspaces}

	A \emph{hyperspace} of a metric space $(X,\dist)$ is a specified family of nonempty closed subsets of $X$. 
	The hyperspaces which will be of our interest are $2^X$, the hyperspace of all nonempty compact subsets of $X$ and $\mathcal{C}(X)$ which is the hyperspace of all continua, i.e. the connected elements of $2^X$.
	For a point $x\in X$ and a nonempty subset $A\subset X$ we define the \emph{distance from the point $x$ to the set $A$} as $\dist(x,A)=\inf\{\dist(x,y)\colon y\in A\}$
	and for each nonempty set $A\subset X$ and each $\eps>0$, we define the \emph{$\eps$-neighborhood of the set $A$} as $N(A,\eps)=\{x\in X\colon \dist(x,A)<\eps\}$.
	We can now endow $2^X$ with a function $\dist_H\colon 2^X\times 2^X\to\mathbb{R}$, defined for each pair $A,B\in 2^X$ as:
	\[\dist_H(A,B)=\inf\{\eps\geq 0\colon {A\subset N(B,\eps)}\text{ and }B\subset N(A,\eps)\}.\] 
	It is well known that $\dist_H$ is a metric, called \emph{the Hausdorff metric}, and that $(2^X,\dist_H)$ is a compact metric space (e.g., see \cite{Nadler}). 
	It is also well known that $\mathcal{C}(X)$ is a compact subset of $2^X$. 
	The topology generated by $\dist_H$ does not depend on the metric $\dist$ but only on the topology that $\dist$ generates on $X$.

	If $f\colon X\to X$ is a continuous map, then it naturally extends to $\bar{f}\colon 2^X\to 2^X$ via $\bar{f}(A)=f(A)$.
	It is well known that $\bar{f}$ defined this way is a continuous map on $2^X.$
	Obviously, $\mathcal{C}(X)$ is $\bar{f}$-invariant subset of $2^X$ and we may denote $\tilde{f}=\bar{f}\vert_{\mathcal{C}(X)}$.
	This way we obtain two \emph{induced systems} $(2^X,\bar{f})$ and $(\mathcal{C}(X),\tilde{f})$ generated by $(X,f)$.
	Observe that $(X,f)$ may be considered a subsystem (that is, a closed invariant subset) of $(\mathcal{C}(X),\tilde{f})$ by identifying $x\in X$ with $\{x\}\in \mathcal{C}(X)$,
	and thus also a subsystem of $(2^X,\bar f)$. 
	{Note that $\mathcal{O}_f(A)\neq \mathcal{O}_{\tilde f}(A)$, because the first orbit is a subset of $X$, while the second one is a set of subsets of $X$.}
 Finally, we say that $A\in\mathcal{C}(X)$ is \emph{asymptotically degenerate} if $\diam f^n(A)\to 0$ when $n\to\infty$.

Throughout the paper, we will be using the following result, which is a direct consequence of~\cite[Theorem 1.1]{Jelic2}.
\begin{theorem}\label{thm:JO2}
Let $f\colon I\to I$ be an interval map.
Then, for every $A\in\mathcal{C}(I)$  (at least) one of the following properties holds:
\begin{enumerate}[(i)]
\item\label{asy_per} $A$ is asymptotically periodic, or
\item\label{wandering} $A$ is wandering, i.e. $f^k(A)\cap f^j(A)=\emptyset$ for all $j\neq k$.
\end{enumerate}
\end{theorem}
\section{Preparatory lemmas}\label{sec:3}

\begin{lemma}\label{lem:nofix}
Let $I=[0,1]$, $a,b\in I$ and $f\colon I\to I$ continuous.
If $(a,b)\cap\Fix(f)=\emptyset$ then $f(x)< x$ for all $x\in (a,b)$ or $f(x)>x$ for all $x\in (a,b)$.
\end{lemma}
\begin{proof}
    Follows from continuity of $g\colon I\to \mathbb{R}$, $g(x)=f(x)-x$.
\end{proof}

\begin{lemma}\label{lem:unique_fixed}
    Let $f\colon I\to I$ be a continuous map of $I=[0,1]$.
    Let $A\in\mathcal{C}(I)$ be a fixed point of $\tilde{f}$.
    If there is a unique fixed point $z\in A$ such that $\omega_f(x)=\{z\}$, for all $x\in A$, then $A=\{z\}$.
\end{lemma}

\begin{proof}
    Suppose on the contrary, that $A$ is nondegenerate.
    Therefore, $A$ has endpoints $\{e_i\colon 1\leq i\leq m\}$ different than $z$ for $m=1$ or $m=2$, depending on whether $z$ is an endpoint of $A$ or not.

    Denote by $\{L_i\colon 1\leq i\leq m\}$ the subarcs of $A$, each $L_i$ having endpoints $z$ and $e_i$.
    Note that the assumptions on $A$ stand if we replace $f$ with $f^n$ for each $n\geq 1$.
     In particular, for every $n\geq 0$ and $i,\ 1\leq i\leq m$, $\Int L_i\cap\Fix\left(f^ n\right)=\emptyset$.
    By this, and by Lemma~\ref{lem:nofix}, $f^n(L_i)\not\supset L_i$ for all $n\geq 1$.

    On the other hand, since $f(A)=A$, we either have $f(L_1)\supset L_1$ or $m=2$, $f(L_1)\supset L_2$ and $f(L_2)\supset L_1$ {or $m=2$, $f(L_2)\supset L_2$}.
    In 
    {all }cases, $f^m(L_1)\supset L_1$, contradicting what has just been proven. Indeed, $A=\{z\}$.
    \end{proof}

\begin{lemma}\label{lem:fixed_sub_fixed}
Let $I=[0,1]$, $f\colon I\to I$ continuous and $A,B\in\mathcal{C}(I)$ two fixed points of $\tilde{f}$ such that $A\setminus B\neq\emptyset$.
Then $\Per(f)\cap (A\setminus B)\neq\emptyset$.
\end{lemma}

\begin{proof}
    Suppose on the contrary, that $\Per(f)\cap (A\setminus B){=}\emptyset$. It is easy to see that also $\omega(f\vert_A)\cap (A\setminus B){=}\emptyset$.
    This is because, {by results of Blokh \cite{B_spectral},} basic sets are contained in $\overline{\Per(f)}$ and, for each maximal solenoid $Q_{\max}\subset A$, $Q_{\min}=Q_{\max}\cap\overline{\Per(f)}$ has to be subset of a fixed interval $A\cap B$ and, therefore, $Q_{max}\subset A\cap B$.

    Since $f(B)=B$, we can define a factor system $(I/_{B},f/_{B}),$ where $I/_{B}$ is the factor space and $f/_{B}\colon I/_{B}\to I/_{B} $ is induced map given by $f/_{B}\circ\pi_{B}=\pi_{B}\circ f,$ where $\pi_{B}\colon I\to I/_{B}$ is the canonical projection.
    Note that $\pi_B(A)$ is a subarc of $I/_B$ containing the unique fixed point $\pi_B(B)$ such that $f/_{B}$-trajectories of all the points of $\pi_B(A)$ converge towards $\pi_B(B)$.
    By Lemma~\ref{lem:unique_fixed}, $\pi_B(A)=\pi_B(B)$ and hence $A\subset B$. A contradiction.
\end{proof}

\begin{lemma}\label{lem:intersects_alpha}
	Let $A\in \mathcal{C}(I)$ and let $B\in s\alpha(A).$ Then $B\cap s\alpha(x)\neq\emptyset$ for all $x\in A$. 
\end{lemma}
\begin{proof}
    Let $\{A_n\}_{n\leq 0}$ be a backward branch of $A$ such that $B\in\alpha(\{A_n\}_{n\leq 0})$.
    Then there is a strictly decreasing sequence of negative integers $\{n_k\}_{k\geq 0}$ such that $\lim_k A_{n_k}=B$.

    For every $x_0\in A$ we can find a backward branch $\{x_n\}_{n\leq 0}$ of $x_0$ such that $x_n\in A_n$ for all $n\leq 0$.
    Let $\{n_{k_j}\}_{j\geq 0}$ be a subsequence of $\{n_k\}_{k\geq 0}$ such that $\{x_{n_{k_j}}\}_{j\geq 0}$ converges.
    Then $\lim_j x_{n_{k_j}}\in s\alpha(x)\cap B$.
\end{proof}

\section{Periodic nondegenerate intervals}\label{sec:4}

The goal of this section is to prove Lemma~\ref{lem:asympt_per}, i.e. that every nondegenerate and asymptotically periodic element of an $\alpha$-limit set of a backward branch in the induced system is in fact periodic.

\begin{lemma}\label{lem:intersects_orbit}
    Let $\{A_n\}_{n\leq 0}$ be a backward branch of $A=A_0\in\mathcal{C}(I)$ and suppose that there is some nondegenerate $B\in \alpha(\{A_n\}_{n\leq 0})$.
    Then $A_n\cap \mathcal{O}(B)\neq \emptyset$ for all $n\leq 0$.
    {Furthermore, every $C\in \alpha(\{A_n\}_{n\leq 0})$ intersects some element of $\omega_{\tilde{f}}(B)$.}
\end{lemma}

\begin{proof}
    Let $\eps=(\diam B)/2>0$. Then $\dist_H(A_n,B)<\eps$ implies $A_n\cap B\neq\emptyset$. But $B\in \alpha(\{A_n\}_{n\leq 0})$ so $A_n\cap B\neq\emptyset$ for infinitely many $n\leq 0$.
    Therefore, $A_n\cap\mathcal{O}(B)\neq\emptyset$ for all $n\leq 0$ and first part of the lemma  is proved.

    To prove the second statement, take some $C\in \alpha(\{A_n\}_{n\leq 0})$ and suppose that $C\cap(\bigcup\omega_{\tilde{f}}(B))=\emptyset$.
    But $\omega_{\tilde{f}}(B)$ is a compact set in $2^X$ and therefore $\bigcup\omega_{\tilde{f}}(B)$ is a compact subset of $I$.
    Hence, we can find some $\eps>0$ such that \[\N\left(\bigcup\omega_{\tilde{f}}(B),\eps\right)\cap C=\emptyset.\]
    Now take some $y\in\Int B$. $\omega_f(y)\subset \bigcup\omega_{\tilde{f}}(B)$ so there is some $m\geq 0$ such that $f^n(y)\in \N(\bigcup\omega_{\tilde{f}}(B),\eps)$ for all $n\geq m$.
    There are also some $j\leq 0$ such that $A_j\cap \N(\bigcup\omega_{\tilde{f}}(B),\eps)=\emptyset$ (since $C\in \alpha(\{A_n\}_{n\leq 0})$) and some $i<j-m$ such that $y\in A_i$ (since $B\in \alpha(\{A_n\}_{n\leq 0})$). But $j-i>m$ and hence $f^{j-i}(y)\in A_j\cap \N(\bigcup\omega_{\tilde{f}}(B),\eps)$. A contradiction.   
\end{proof}

\begin{lemma}\label{lem:intersects_omega}
	Let $A\in \mathcal{C}(I)$ and let $B\in s\alpha(A)$ be nondegenerate. If $\omega_{\tilde{f}}(B)$ is a periodic orbit of period $p$ then $B\cap\lim_n f^{np}(B)\neq\emptyset$.
\end{lemma}	

\begin{proof}
	Since $s\alpha(\tilde{f})=s\alpha(\tilde{f}^p)$, we can consider $f^p$ instead of $f$, i.e. for the sake of simplicity we can assume that $p=1$.
Denote $F=\lim_n f^{n}(B)$ and let us prove that $B\cap F\neq\emptyset$.
	Let $\{A_n\}_{n\leq 0}$ be some backward branch of $A$ such that $B\in\alpha\left(\{A_n\}_{n\leq 0}\right)$. 
 Since $\omega_{\tilde{f}}(B)=\{F\}$, by Lemma~\ref{lem:intersects_orbit}, $B\cap F\neq\emptyset$ and the lemma is proved.
\end{proof}

\begin{lemma}\label{lem:A_not_sub_BUF}
	Let $A\in \mathcal{C}(I)$ and let $B\in\alpha\left(\{A_i\}_{i\leq 0}\right)$ for some backward branch $\{A_i\}_{i\leq 0}$ of $A$.
	Suppose there is some $F\in \mathcal{C}(I)$ such that $\omega_{\tilde{f}}(B)=\{F\}$ and additionally $B\setminus F\neq\emptyset$.
	Then $A_i\subset B\cup F$ for at most finitely many $i\leq 0$.
\end{lemma}

\begin{proof}
Suppose on the contrary, i.e. that $A_i\subset B\cup F$ for infinitely many $i\leq 0$.
	There is some $\eps>0$ such that $B\setminus \N(F,\eps)\neq\emptyset$.
	Since $\lim f^n(B\cup F)=F$, there is some $n_0\geq 0$ such that $f^n(F\cup B)\subset \N(F,\eps)$ for all $n\geq n_0$.
	Furthermore, since $B\in\alpha\left(\{A_n\}_{n\leq 0}\right)$ and by the assumption, there are $i_1,i_2<0$, $i_2<i_1-n_0$ such that $A_{i_1}\setminus  \N(F,\eps) \neq\emptyset$ and $A_{i_2}\subset B\cup F$.
	But ${i_1}-i_2>n_0$ and therefore $f^{{i_1}-i_2}(B\cup F)\subset \N(F,\eps)$ so \[A_{i_1}=f^{{i_1}-i_2}(A_{i_2})\subset \N(F,\eps),\] a contradiction.
\end{proof}

\begin{lemma}\label{lem:B_not_sub_Ai}
	Let $A\in \mathcal{C}(I)$ an let $B\in\alpha\left(\{A_i\}_{i\leq 0}\right)$ for some backward branch $\{A_i\}_{i\leq 0}$ of $A$.
	Suppose there is some $F\in \mathcal{C}(I)$ such that $\omega_{\tilde{f}}(B)=\{F\}$ and additionally $F\setminus B\neq\emptyset$.
	Then $B\subset A_i$ for at most finitely many $i\leq 0$.
\end{lemma}				
		
\begin{proof}
	Suppose on the contrary, i.e. that $B\subset A_i$ for infinitely many $i\leq 0$.
	Take some $\eps>0$ such that $E=F\setminus \N(B,\eps)\neq\emptyset$.
	Since $\lim f^n(B)=F$, there is some $n_0\geq 0$ such that $f^n(B)\cap E\neq\emptyset$ for all $n\geq n_0$.
	There are also some $i_1,i_2$, $i_1<i_2-n_0<0$ such that $A_{i_2}\subset \N(B,\eps)$ and $B\subset A_{i_1}$.
	But then $\N(B,\eps)\supset A_{i_2}=f^{i_2-i_1}(A_{i_1})\supset f^{i_2-i_1}(B)$.
	On the other hand $i_2-i_1>n_0$ so $f^{i_2-i_1}(B)\cap E\neq \emptyset$.
	This implies that \[\N(B,\eps)\cap \left(F\setminus \N(B,\eps)\right)\neq\emptyset,\] a contradiction.
	Indeed, $B\subset A_i$ for at most finitely many $i\leq 0$.
\end{proof}

\begin{lemma}\label{lem:asympt_per}
Let $B\in s\alpha(\tilde{f})$ be nondegenerate and asymptotically periodic.
Then $B$ is periodic.
\end{lemma}	
	
\begin{proof}
 Take some nondegenerate and asymptotically periodic $B\in s\alpha(\tilde{f})$ and suppose it is not periodic.
	We can assume that $\omega_{\tilde{f}}(B)$ consists of a single continuum $F$,
	in particular $B\neq F$.
	Denote $B=[a,b]$ and $F=[c,d]$ with $a<b$ and $c\leq d$.
By Lemma~\ref{lem:intersects_omega}, $B\cap F\neq\emptyset$.
	Take some $A\in \mathcal{C}(I)$ with a backward branch $\{A_n\}_{n\leq 0}$ such that $B\in\alpha\left(\{A_n\}_{n\leq 0}\right)$.
	Then also $F\in \alpha\left(\{A_n\}_{n\leq 0}\right)$.
Without loss of generality, we may assume that one of the following cases is true.

\medskip\paragraph{\textbf{Case 1:} $F\subset \Int B$, i.e. $a<c\leq d<b$.} 
Since $F\in \alpha\left(\{A_n\}_{n\leq 0}\right)$, there is a strictly decreasing sequence $\{n_k\}_{k\geq 0}$ of negative integers such that $A_{n_k}\subset B=B\cup F$ for all $k\geq 0$. This contradicts Lemma~\ref{lem:A_not_sub_BUF}.

\medskip\paragraph{\textbf{Case 2:} $B\subset \Int F$, i.e. $c<a<b<d$.}
Since $F\in \alpha\left(\{A_n\}_{n\leq 0}\right)$, there is a strictly decreasing sequence $\{n_k\}_{k\geq 0}$ of negative integers such that $B\subset A_{n_k}$ for all $k\geq 0$. This contradicts Lemma~\ref{lem:B_not_sub_Ai}.

\medskip\paragraph{\textbf{Case 3:} $a<c<b=d$ or $a<b=c=d$ or $a<c<b<d$ or $a<b=c<d$.}  
There is a sequence $n_k<0$ such that $A_{n_k}\to F$ as $k\to\infty$. Denote $A_{n_k}=[e_{n_k},f_{n_k}]$.
We can assume that $a<e_{n_k}$ for every $k\geq 0$.
	Then, by Lemma~\ref{lem:A_not_sub_BUF}, we can also assume that 
    $d<f_{n_k}$ for every $k\geq 0$ and that $f_{n_k}$ form a strictly monotone sequence converging towards $d$. 
	We want to prove that $f([d,f_{n_k}])\supset [d,f_{n_k}]$ for almost all $k\geq 0$.
	Let $\eps=c-a>0$.
	
	First we claim that for all $x\in B$ and for all $n\geq 0$, $f^n(x)\leq d$.
	To see this, suppose on the contrary, that $f^m(x)>d$ for some $m\geq 0$ and some $x\in B$.
	Then $B^\prime= B\cup F\cup f^{m}(B)\supset F$. Note that $B^\prime$ is connected by the assumptions of Case 3. Also $\lim_{n\to \infty} f^n(B^\prime)=F$, so there is some $k_0$ such that $f^k(B^\prime)\subset \N(F,\frac{\eps}{2})$ for all $k\geq k_0$.
	Since $F,B\in\alpha\left(\{A_n\}_{n\leq 0}\right)$, {and $F\subset \Int B^\prime$} there are $j_1,j_2$, {$n_{j_1}<n_{j_2}-k_0<0$} such that {$A_{n_{j_1}}\subset 
 B^\prime$}
 and $A_{n_{j_2}}\subset \N(B,\eps/2)$.
	But then $f^{n_{j_2}-n_{j_1}}(A_{n_{j_1}})=A_{n_{j_2}}\subset \N(F,\eps/2)$, contradicting the choice of $A_{n_{j_2}}$.
	Therefore, indeed the claim holds, that is, $f^n(x)\leq d$ for all $x\in B$ and for all $n\geq 0$.

	Since $f(F)=F$, there is some $\delta>0$ such that $f(\N(F,\delta))\subset \N(F,\eps)$.
	Take some $j_0\geq 1$ such that for all $j\geq j_0$ we have $f_{n_j}\in \N(F,\delta)$.
    Fix any $j\geq j_0$ and put $r=n_{j-1}-n_j>0$.
	
		There is some $y\in A_{n_j}$ such that $f^r(y)=f_{n_{j-1}}$.
		Moreover, there is also $s^\prime\geq 0$ such that $f^k(y)\in( d,f_{n_j}]$ for all $0\leq k\leq s^\prime$ and $f^{s^\prime+1}(y)\notin( d,f_{n_j}]$ {because the sequence $f_{n_k}$ is strictly monotone} {and because $f^n(x)\leq d$ for all $x\in [e_{n_j},d]$ and every $n\geq 0$}.  
		Obviously $f^{s^\prime+1}(y)\notin F$ and if $f^{s^\prime+1}(y)\in B$ then, {since $r>s^\prime$, we have} $f^r(y)\leq d<f_{n_{j-1}}$ {which is a contradiction}.
		But $f(( d,f_{n_j}])\subset \N(F,\eps)$ and the above arguments lead to $f^{s^\prime+1}(y)>f_{n_j}$.
		Now, combining $f([d,f_{n_j}])\cap F\neq \emptyset$ with $f([d,f_{n_j}])\ni f^{s^\prime+1}(y)>f_{n_j}$ we indeed obtain that $f([d,f_{n_j}])\supset [d,f_{n_j}]$ for some $j_0> 0$ and all $j\geq j_0$.

	Denote $\kappa=\min\{\eps/2,\dist(d,f_{n_{j_0}})\}$.
	Since $f^n(B\cup F)\to F$ as $n\to\infty$, there is some $n_0\geq 0$ such that $f^n(B\cup F)\subset \N(F,\kappa)$ for all $n\geq n_0$.
	There are some $i_1,i_2,i_3$ such that ${0>n_{i_3}>n_{i_2}>n_{i_1}+n_0}$ and such that $\dist_H(F,A_{n_{i_3}})<\kappa,\ \dist_H(B,A_{n_{i_2}})<\kappa,\ \dist_H(F,A_{n_{i_1}})<\kappa$.
	In particular, $A_{n_{i_3}}\cup A_{n_{i_1}}\subset \N(F,\kappa)$ and $A_{n_{i_2}}\setminus \N(F,\kappa)\neq\emptyset$.
	Let $j_1=\min\{j\geq 0\colon f_{n_j}\in A_{n_{i_1}}\}$. {Note that $j_1>j_0$ since the sequence $\{f_{n_k}\}_{k\geq 0}$ is monotone and by the choice of $\kappa$ and $i_1$.}
	Take some $n_{i_0}<n_{i_1}$ such that $A_{n_{i_0}}\subset \N(F,\dist(d,f_{n_{j_1}}))$.
	So there is some $z\in A_{n_{i_0}}$ such that 
    {$f^{n_{i_2}-n_{i_0}}(z)\in A_{n_{i_2}}\setminus \N(F,\kappa)$}.
 But $n_{i_2}-n_{i_0}>n_0$ so {$z\not\in B\cup F$ and as a consequence }
    $z\in( d,f_{n_{
    i_0}}]$.
	Since $f^{n+1}(( d,f_{n_{
     i_0
    }}])\supset f^{n}(( d,f_{n_{
     i_0}}])$ for every $n\geq 0$ and since $( d,f_{n_{
    i_0
    }}]\subset A_{n_{i_0}}$, 
    we have that $f^{n_{i_2}-n_{i_0}}(z)\in A_{n_{i_2}}\setminus \N(F,\kappa)$ implies $f^{{
    n_{i_2}}-n_{i_0}}(z)\in A_{n_{i_3}}\setminus \N(F,\kappa)$,     
     which is a contradiction with $A_{n_{i_3}}\subset \N(F,\kappa)$.
	
\medskip\paragraph{\textbf{Case 4:} $c<a<b=d$.}	
	Note that $f(F)=F$ so first suppose that there are some $y\in \Int F$ and some $m\geq 0$ such that $f^m(y)=d$.
	Denote $\eps=\dist(a,c)/2$.
	Since $f^m(F)=F$, there is some $\delta>0$ such that $\dist_H(D,F)<\delta$ implies $\dist_H(f^m(D),F)<\eps$ {and $y\in D$}.
	There is strictly decreasing sequence of negative integers $n_k$ such that $\dist_H(A_{n_k},F)<\delta$ and {so} $y\in A_{n_k}$ for all $k\geq 0$.
	But then $B\subset f^m(A_{n_k})=A_{n_k+m}$ for all $k\geq 0$, contradicting Lemma~\ref{lem:B_not_sub_Ai}.
	We conclude that $d\notin\bigcup_{n\geq 0}f^n(\Int F)$.
	In particular we have $f(d)=d$ or $f(c)=d$.
	
	Suppose $f(c)=d$.
	If $f(y)=c$ for some $y\in \Int F$ then $f^2(y)=d$, contradicting what has already been proven so $f(d)=c$.
	Hence, either $f(d)=d$ or $f(d)=c$ and $f(c)=f(d)$ so in both cases we have $f^2(d)=d$.
	{By the property} $s\alpha(\tilde{f})=s\alpha(\tilde{f}^2)$ and  $\omega_{\tilde{f^2}}(B)=\{F\}$, we can replace $f$ with $f^2$, that is, we may assume that $f(d)=d$ holds. {After this modification, the image $f(c)$ of $c$ is not completely clear anymore. We will analyze it later}. 
	
	Let us now prove that $A_i\subset F$ for all $i\leq 0$.
	Suppose on the contrary that $A_r\setminus F\neq\emptyset$ for some $r\leq 0$.
	Then $A_i\setminus F\neq\emptyset$ for all $i\leq r$.
	Hence, whenever $j\leq r$ and $\dist_H(A_j,B)<\eps$, necessarily $d\in A_j$. But then $d\in A_t$ for all  $j\leq t\leq 0$.
	There are infinitely many such $j\leq 0$ so we conclude that $d\in A_i$ for all $i\leq 0$.
	But now $\dist_H(A_i,F)<\eps$ implies $B\subset A_i$ so, by Lemma~\ref{lem:B_not_sub_Ai}, $B\notin\alpha(\{A_i\}_{i\leq 0})$, which is a contradiction.
	
	So far we have proven that $A_i\subset F$ for all $i\leq 0$ and $f(d)=d$. Let us now prove that $f(c)=c$. 
	Suppose on the contrary, that $f(c)\neq c$.
	Since $f(F)=F$ and since $f(d)=d$, it follows that $f(z)=c$ for some $z\in\Int F$	.
	Let $\eps_1=\min\{\dist(c,z),\dist(d,z)\}>0$.
	There is some $0<\delta_1<\eps_1$ such that $\dist_H(F,D)<\delta_1$ implies $\dist_H(F,f(D))<\eps_1$.
	There is a strictly decreasing sequence $\{m_k\}_{k\geq 0}$ of negative integers such that $\dist_H(F,A_{m_k})<\delta_1$ for all $k\geq 0$.
	But then $A_{m_k+1}$ and $A_{m_k+2}$ both lie in $F$ and contain $c$ so, for all $k\geq 0$, either \[A_{m_k+2}=f(A_{m_k+1})\subset A_{m_k+1}\quad\text{or}\quad A_{m_k+2}=f(A_{m_k+1})\supset A_{m_k+1}.\] 
	This implies that $\{A_i\}_{i\leq 0}$ form a nested sequence.
	But this is in contradiction with $B,F\in\alpha\{A_i\}_{i\leq 0}$.
	Indeed, $f(c)=c$.
	
	If there was a fixed point $w$ in $( c,a) $, using $F\in\alpha(\{A_i\}_{i\leq 0})$, we easily see that $w\in A_i$ for all $i\leq 0$, contradicting $B\in\alpha(\{A_i\}_{i\leq 0})$, so $\Fix(f)\cap( c,a)=\emptyset$.
	This, combined with Lemma~\ref{lem:nofix}, implies that either $f(w)\in( w,d)$ for all $w\in ( c,a)$ or $f(w)\in( c,w)$ for all $w\in ( c,a)$.
	
	First consider the case $f(w)\in( c,w)$ for all $w\in ( c,a)$.
	Take some $w_1\in ( c,a)$.
If $A_i\cap( c,w_1)\neq\emptyset$ then $A_j\cap( c,w_1)\neq\emptyset$ for all $i\geq j\geq 0$.
	But $A_i\cap( c,w_1)\neq\emptyset$ for infinitely many $i\leq 0$ so	$A_i\cap( c,w_1)\neq\emptyset$ for all $i\leq 0.$
	This is in contradiction with $B\in\alpha(\{A_i\}_{i\leq 0})$.
	
	It remains to consider the case $f(w)\in( w,d)$ for all $w\in ( c,a)$.
    {Since $\omega_{\tilde{f}}(B)=\{F\}$, we cannot have $\tilde{f}(B)\subset B$ and therefore }
$f(B)=[a_1,d]$ for some $c<a_1<a$. But $f([a_1,a])\subset[a_1,d]$ so $f([a_1,d])=[a_1,d]$.
	Since {we already proved that $A_i\subset F$ {and since $B\in \alpha(\{A_i\}_{i\leq 0})$}, we have }$A_i\subset [a_1,d]$ for infinitely many $i\leq 0$, {and hence} we conclude that $A_i\subset [a_1,d]$ for all $i\leq 0$.
	This contradicts $F\in \alpha(\{A_i\}_{i\leq 0})$ 	
\end{proof}	

\section{Wandering intervals}\label{sec:5}

In this section we consider the case when $\alpha$-limit set of a backward branch in the induced system contains a nondegenerate wandering interval.
The first goal is to present Example~\ref{ex:wandering}, showing that this case is in fact possible.
Later we continue by proving that all such cases are very similar to the presented example.

Let us denote by $T$ the standard tent map $T\colon[0,1]\to[0,1]$,

\begin{equation*}
  T(x)=\begin{cases}
    2x, & \text{for $0\leq x\leq\frac{1}{2}$}.\\
    2(1-x), & \text{for $\frac{1}{2}<x\leq 1$ }.
  \end{cases}
\end{equation*}

The following properties of map $T$ are well known.
 
\begin{proposition}\label{prop:tent}\leavevmode\\[-5pt]
\begin{enumerate}
\item For every $n\in\mathbb{N}$ the n-th iterate of $T$ is defined by
\begin{equation*}
  T^n(x)=\begin{cases}
    2^n\left(x-\frac{i}{2^{n-1}}\right), & \text{if $ x\in\left[\frac{2i}{2^n},\frac{2i+1}{2^n}\right]$}.\\
     -2^n\left(x-\frac{i+1}{2^{n-1}}\right), & \text{if $ x\in\left[\frac{2i+1}{2^n},\frac{2i+2}{2^n}\right]$ }
  \end{cases} \ \text{for $i=0,1,\ldots,2^{n-1}-1$.}
\end{equation*}
\item Point $x\in[0,1]$ is periodic under $T$ if and only if $x$ is a rational number of the form $m/p$, where $m$ is an even positive integer and $p$ is an odd positive integer with $m<p$.
\item Point $x\in[0,1]$ is rational if and only if $x$ is eventually periodic under $T$.
\end{enumerate} 
\end{proposition}

Let us denote $S=\bigcup_{n\geq 0}\left\{\frac{i}{2^n}\colon 0\leq i\leq 2^n\right\}$.

\begin{lemma}\label{lem:tent_1}
Let $x\in [0,1]$ be a transitive point of $T$. For every interval $[a,b]$ such that $0<a<x<b<1$, there is some $n\in\mathbb{N}$ and $[c,d]$ such that $(a+x)/2<c<x<d<(x+b)/2$ and $T^{n}([c,d])=[a,b]$.  
\end{lemma}

\begin{proof}
There is an $m\geq 2$ such that $\frac{1}{2^{m-1}}<\min\{\frac{x-a}{2},\frac{b-x}{2}\}$.
Take some $n\geq m$ such that $T^{n}(x)\in ( a,b)$.
Since $x\notin S$, there is some $k$, $0\leq k< 2^{n-1}$, such that $\frac{a+x}{2}<\frac{k}{2^{n-1}}<x<\frac{k+1}{2^{n-1}}<\frac{x+b}{2}$.
Furthermore, either $x\in \left(\frac{k}{2^{n-1}},\frac{2k+1}{2^n}\right)$ or $x\in \left(\frac{2k+1}{2^n}, \frac{k+1}{2^{n-1}}\right)$ so, without loss of generality, suppose the former and denote $J= \left[\frac{k}{2^{n-1}},\frac{2k+1}{2^n}\right]$.
By definition $x\in J\subset \left(\frac{a+x}{2},\frac{x+b}{2}\right)$.

Proposition~\ref{prop:tent}
implies $T^n(J)=[0,1]$ and $T^n\vert_J$ is injective.
Since $T^n(x)\in[a,b]$, it is enough to put $[c,d]\coloneqq\left(T^n\vert_J\right)^{-1}\left([a,b]\right)$.
\end{proof}

\begin{figure}[h]
\centering
\includegraphics[scale=1.15]{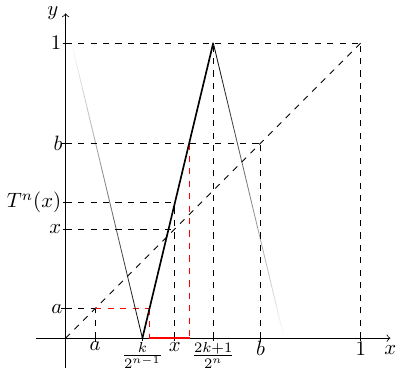}
\caption{Red colored interval represents $[c,d]$ from Lemma~\ref{lem:tent_1}.}
\end{figure}

\begin{lemma}\label{lem:tent_2}
Let $x\in[0,1]$ be a transitive point of $T$. For every interval $I_0=[a,b]$ such that $0<a<x<b<1$ there is a sequence $\{I_n\}_{n\geq 1}$ of compact intervals and a sequence $\{j_n\}_{n\geq 1}$ in $\mathbb{N}$ such that $I_{n+1}\subset I_n$ for all $n\geq 0$, $\bigcap_{n\geq 0} I_n=\{x\}$ and $T^{j_n}(I_n)=I_{n-1}$ for all $n\geq 1$.
\end{lemma}
\begin{proof}
The claim follows by an easy induction using Lemma~\ref{lem:tent_1}.
\end{proof}

\begin{example}\label{ex:wandering}
Let $I=[0,1]$ and let $T\colon I\to I$ be the standard tent map.
Fix some transitive point $x$ of $T$.
 Denote $R=\{y\in I\colon x\in\mathcal{O}_T(y)\}\cup\mathcal{O}_T(x)$.
 Note that $R$ is forward and backward invariant, countable subset of $I$.
 We will perform a kind of Denjoy construction (see~\cite{Devaney} for more details).
 To that end, we ``replace" each $y\in R$ by an interval $I_y$ of length $d_y$, keeping the sum $\sum_{y\in R} d_y$ finite. For $y\in I\setminus R$ set $I_y=\{y\}$.
 This way we obtain new arc $J=\bigcup_{y\in I}I_y$.
 
 We can naturally extend $T$ to a continuous map $f\colon J\to J$ such that $f(I_y)=I_{T(y)}$ for every $y\in[0,1]$.
 Namely, we map $I_y$ to $I_{T(y)}$ by increasing homeomorphism when $y\in[0,\frac{1}{2}]$ and by decreasing homeomorphism otherwise. 
 Clearly, $T$ is semi-conjugate to $f$
 via a factor map $\pi\colon J\to I$ which projects every $I_{y}$ to $\{y\}$.
 
 Fix any $A\in \mathcal{C}(J)$ such that $I_x\subset \Int A$. 
 Then $\pi(A)$ is a subinterval of $I$ with $x\in\Int \pi(A)$.
 By Lemma~\ref{lem:tent_2}, there is a backward branch $\{B_n\}_{n\leq 0}$ of $B_0=\pi(A)$ such that $\{x\}\in\alpha\left(\{B_n\}_{n\leq0}\right)$.
 Therefore, $I_x$ is an asymptotically degenerate, nondegenerate subinterval of $J$ such that $I_x\in\alpha\left(\{\pi^{-1}(B_n)\}_{n\leq 0}\right)$. 
 \end{example}

The following Lemma~\ref{lem:wandering} provides a description on all the possible situations in which a backward branch in the induced system has a nondegenerate wandering interval in its $\alpha$-limit set. 

\begin{lemma}\label{lem:wandering}
	Let $\{A_n\}_{n\leq 0}$ be a backward branch of some $A\in\mathcal{C}(I)$ and let $B\in\alpha(\{A_n\}_{n\leq 0})$ be nondegenerate and wandering, i.e. $f^i(B)\cap f^j(B)=\emptyset$ for all $i\neq j$.  
    Then there is some basic set $\omega=B(K,f)$ such that:
    \begin{enumerate}
        \item $B\cap\omega\neq\emptyset $,
        \item $B\cap\omega(f)=(\partial B\cap \omega)\setminus \Per(f)$
        \item $B\subset K$
    \end{enumerate}
 Furthermore, let $\mathcal{K}=\{K_{(i)},i=1,2,\ldots, m\}$ be the set of intervals described in Remark~\ref{rem:model_BS}.
Then, for every $n\leq 0$, $A_n\subset K_{(i_n)}$ for some $1\leq i_n\leq m$ and $\omega_{\tilde{f}}(A)=\mathcal{K}$.
\end{lemma}	
\begin{proof}
	 Suppose there is a backward branch $\{A_n\}_{n\leq 0}$ of $A\in\mathcal{C}(I)$ with wandering and nondegenerate $B\in\alpha\left(\{A_n\}_{n\leq 0}\right)$.
  We claim that, for all $i\leq 0$, $f^n(B)\subset A_i$ for infinitely many $n\geq 0$.
	 To prove this, fix some $i\leq 0$.
	Since $B\in\alpha\left(\{A_n\}_{n\leq 0}\right)$, there is some strictly increasing sequence $\{n_k\}_{k\geq 0}$ of positive integers such that $A_{i-n_k}\cap B\neq\emptyset$ for all $k\geq 0$.
	But then, since $f^{n_k}(A_{i-n_k})=A_i$, we get that $A_i\cap f^{n_k}(B)\neq\emptyset$ for all $k\geq 0$.
	But, since $f^{n_k}(B)\cap f^{n_j}(B)=\emptyset$ for $j\neq k$, $f^{n_k}(B)\subset {A_i}$ for all but at most two distinct $k\geq 0$.
	Indeed, for any $i\leq 0$, $f^n(B)\subset A_i$ for infinitely many $n\geq 0$.
	
	By Lemma~\ref{lem:intersects_alpha}, $B\cap s\alpha(\tilde{f})\neq\emptyset$ and, in particular, $B\cap \omega(f)\neq\emptyset$.
	But $B$ is wandering, so $B\cap\Per(f)=\emptyset$.
	Therefore, since $B$ is asymptotically degenerate, $B\cap \omega(f)$ lies in a single maximal $\omega$-limit set.
	By {Remark~\ref{rem:alpha}}, we have two different cases:
	\begin{enumerate}
	\item  all the points in $B\cap \omega(f)$ are non-periodic and belong to a single maximal basic set, or
	\item all the points in $B\cap \omega(f)$ belong to a {maximal} solenoid. 
	\end{enumerate}
	Take some $y\in B\cap \omega(f)$.
	
	{Firstly we consider the case $y\in Q_{max}$, where $Q_{max}\subset Q=\cap_n M_n$ is a maximal solenoid and $\{M_n\}_{n\geq 0}$ some nested sequence of cycles of intervals.}
	First let us prove that for every $i\leq 0$, $A_i\subset Q$.
	Suppose on the contrary, that there are some $r\geq 0$ and some $i_1\leq 0$
such that $A_{i_1}\setminus M_r\neq \emptyset$.
	Take some $z\in Q_{min}\cap \Int M_r$.
 $B$ is asymptotically degenerate and $z\in\omega_f(y)$, so $\{z\}\in\omega_{\tilde{f}}(B)$.
 Therefore, there is some $m\geq 0$ such that $f^m(B)\subset\Int M_r$.
 This, together with the fact that $f^m(B)\in\alpha\left(\{A_i\}_{i\leq 0}\right)$ implies that there is some $i_2<i_1$ such that $A_{i_2}\subset M_r$.
 But then $f^{i_1-i_2}(A_{i_2})=A_{i_1}\subset M_r$, a contradiction.
 Indeed, for all $i\leq 0$, $A_i\subset Q$ and hence $A_i$ are wandering.

	But $B\in\alpha\left(\{A_i\}_{i\leq 0}\right)$ and $B$ is nondegenerate so for all $j,k$ such that $\dist(A_j,B)\leq\diam B/2$ and $\dist(A_k,B)\leq\diam B/2$, $A_j\cap A_k\neq\emptyset$, which is a contradiction.
Therefore, there is a maximal basic set $\omega=B(K,f)$ such that $B\cap \omega\neq\emptyset$ and $B\cap\omega(f)\subset\omega\setminus \Per(f)$.

 Denote by $\mathcal{K}=\{K_{(i)}\colon 1\leq i\leq m\}$ a set of subintervals described in Remark~\ref{rem:model_BS}.
    The property $\omega\cap B\subset\partial B$ easily follows from the facts that $B$ is wandering and that periodic points are dense in $\omega$. Let us show that $B\subset K$.
Take some $z\in\omega\cap B$. Note that $z$ has an infinite orbit, since otherwise $B$ would not be wandering.
But $B$ is asymptotically degenerate and $\mathcal{O}_f(z)$ is an infinite subset of $K$ so there is an $r\geq 0$ such that $f^r(B)\subset\Int K_i$ for some $1\leq i\leq m$.
But $f^r(B)\in \alpha\left(\{A_n\}_{n\leq 0}\right)$ so there is some strictly decreasing sequence $\{n_k\}_{k\geq 0}$ of negative integers such that $\lim_k A_{n_k}=f^r(B)$.
Then  $A_{n_k}$ is contained in some element of $\mathcal{K}$ for almost all $k\geq 0$, but $\mathcal{K}$ is $\tilde{f}$-invariant and hence $A_n$ is contained in some element of $\mathcal{K}$  for every $n\leq 0$.
In particular, $B\subset K$.

\begin{figure}[ht]
\centering
\includegraphics[scale=1]{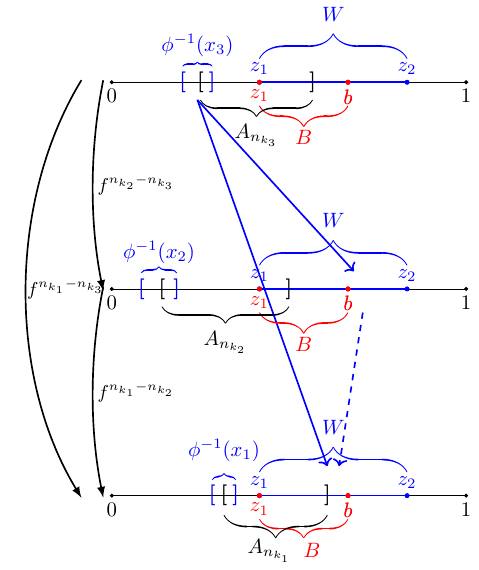}
\label{fig:wandering_model}
\caption{
Final step in the proof of Lemma~\ref{lem:wandering}.
The backward branch elements $A_{n_{k_i}}$ accumulate on $B=[z_1,b]\subsetneq W$ and cover only
part of $W$. Their images $\varphi(A_{n_{k_i}})$ are arcs ending at $\varphi(W)$, so at most
one boundary fibre can fail to lie in $A_{n_{k_i}}$; this forces $W\subset A_{n_{k_1}}$ and hence $B=W$.}
\end{figure}
 
Now let us show that $B\cap\omega=\partial B$.
Suppose on the contrary, that $B\cap\omega=\{z_1\}$ and $z_1\in\partial B$.
Then we may assume that $B=[z_1,b]$ is a proper subset of some wandering interval $W=[z_1,z_2]\subset K$ where also $z_2\in\omega$. 
Now pick a strictly decreasing sequence $\{n_k\}_{k\geq 0}$ of negative integers such that $\lim_k A_{n_k}=B$.
We can assume that $A_{n_{k_1}}\cap A_{n_{k_2}}\neq \emptyset$ and $z_2\not\in A_{n_{k_1}}$ for all $k_1,k_2\geq 0$.
Therefore, no $A_n$ is wandering for any $n\leq 0$ and hence $z_1\in A_{n_k}$ for all $k\geq 0$ and $A_{n}\setminus W\neq \emptyset$ for all $n\leq 0$.
We will prove even more, that $W\subset A_{n_k}$ for almost all $k\geq 0$.
Suppose on the contrary, that $W\setminus A_{n_k}\neq\emptyset$ for infinitely many $k\geq 0$.
 Suppose $W\subset K_{(1)}$ and let $g\colon Y\to Y$ be a mixing interval map such that there is a semi-conjugacy $\phi\colon K_{(1)}\to Y$ between $(K_{(1)},f^m)$ and $(Y,g)$, almost conjugating $f^m\vert_{\omega\cap K_{(1)}}$ and $g$.

 Let $k_1<k_2<k_3$ be such that $W\cap A_{n_{k_i}}\neq\emptyset$ and $W\setminus A_{n_{k_i}}\neq\emptyset$ for $i=1,2,3.$
Note that for every $i=1,2,3$, there can exist at most one $x_i\in\phi(A_{n_{k_i}})$ different from $\phi(W)$ such that $\phi^{-1}(x_i)\not\subset A_{n_{k_i}}$. Moreover, $W=\phi^{-1}(\phi(W))$ is wandering, i.e. for every $n\geq 0$, $f^{n}(W)\cap W=\emptyset$. Recall that every $A_n$ is contained in some element of $\mathcal{K}$. Therefore we must have $f^{n_{k_2}-n_{k_3}}(\phi^{-1}(x_3))=W$ and $f^{n_{k_1}-n_{k_3}}(\phi^{-1}(x_3))=W$.
But then \[f^{n_{k_1}-n_{k_2}}(W)=f^{n_{k_1}-n_{k_2}}(f^{n_{k_2}-n_{k_3}}(\phi^{-1}(x_3)))=f^{n_{k_1}-n_{k_3}}(\phi^{-1}(x_3))=W,\] a contradiction.
Therefore, $B=W$ and hence $B\cap\omega=\partial B$.
We have already proved that $A\subset K_{(i)}$ for some $1\leq i\leq m$ and that $\Int A\cap \partial B\neq\emptyset$.
\end{proof}

\section{Characterization of \texorpdfstring{$\alpha$}{alpha}-limit sets in \texorpdfstring{$\mathcal{C}(I)$}{C(I)}}\label{sec:6}

Goal of this section is to provide full description of $\alpha$-limit sets of backward branches in the induced system.
We start with some preparatory lemmas.

\begin{lemma}\label{lem:periodic_alpha}
    Let $A\in\mathcal{C}(I)$, 
     let $\{A_{n}\}_{n\leq 0}$ be its backward branch {and let $B\in \alpha(\{A_{n}\}_{n\leq 0})$}. If $B\in\Per(\tilde{f})$ is nondegenerate then $\alpha(\{A_{n}\}_{n\leq 0})\cap\Per(\tilde{f})=\mathcal{O}_{\tilde{f}}(B)$. 
\end{lemma}
\begin{proof}
    Suppose on the contrary, that there is some $C\in\alpha(\{A_{n}\}_{n\leq 0})\cap\Per(\tilde{f})\setminus\mathcal{O}_{\tilde{f}}(B)$.
    Take $p\geq 1$ such that $f^p(B)=B$ and $f^p(C)=C$.
    There are strictly decreasing seqences $\{n_k\}_{k\geq 0}$ and $\{m_k\}_{k\geq 0}$ such that $\lim_k A_{n_k}=B$ and $\lim_k A_{m_k}=C$.
    By considering remainders of the elements of these sequences modulo $p$, we see that there are some $0\leq r_1,r_2<p$ and a backward branch $\{A^{\prime}_{n}\}_{n\leq 0}$ of $A$ under $f^p$ such that $f^{r_1}(B)\in \alpha_{f^p}(\{A^{\prime}_{n}\}_{n\leq 0})$ and $f^{r_2}(C)\in \alpha_{f^p}(\{A^{\prime}_{n}\}_{n\leq 0})$.
    But $f^{r_1}(B)$ and $f^{r_2}(C)$ are distinct and fixed by $f^p$, so we can assume that $r_1=r_2=0$, $p=1$ and $A^\prime_n=A_n$ for all $n\leq 0$. 

    Note that $B\cap C\neq\emptyset.$ This is because $\dist_H(A_n,B)\leq\diam B/2$ implies $A_n\cap B\neq\emptyset$, $B$ is fixed and therefore $A_n\cap B\neq\emptyset$ for all $n\leq 0$. This makes $\dist_H(C,B)>0$ impossible so, indeed, $B\cap C\neq\emptyset.$

    Since $B\neq C$, we can assume that $B\setminus C\neq\emptyset$.
    By Lemma~\ref{lem:fixed_sub_fixed}, \[P=\left(\Per(f)\cap B\right)\setminus C\neq\emptyset.\] Note that $P$ is invariant set, since $C$ is invariant.
    
    Let us claim first that $P\cap\Int B=\emptyset$.
    If there was some $z\in P\cap \Int B$ then {if we denote} $\dist_H(\mathcal{O}(z),C)=\delta>0$ {we see that} $\mathcal{O}(z)\cap A_n\neq \emptyset$ for all $n\leq 0$ {which immediately implies } $\dist_H(A_n,C)\geq\delta$ for all $n\leq 0$. {This is a contradiction, so indeed 
 $P\cap B\subset\partial B$.}
 
Let us denote $\eps=\dist_H(B,C)>0$.
\medskip\paragraph{\textbf{Case 1:} $C\not\subset \Int B$}
We can assume {(symmetric case has the same proof)} that $B=[b_1,b_2]$, $C=[c_1,c_2]$, $b_1<b_2$, $c_1\leq c_2$ and $b_2\in C$. {Since $B\setminus C\neq \emptyset$, this implies that $b_1<c_1$.}
By Lemma~\ref{lem:nofix}, either $f(w)\in[b_1,w)$ for all $w\in( b_1,c_1)$ or $f(w)\in( w,b_2]$ for all $w\in( b_1,c_1)$.
The first of these two cases would mean that, for all $w\in( b_1,c_1)$, {if} $w\in A_n$ for some $n\leq -1$ 
{then} $[b_1,w)\cap A_{m}\neq\emptyset$ for all $m\geq n$, contradicting $C\in\alpha(\{A_n\}_{n\leq 0})$.
On the other hand, the remaining case would mean that $[w,b_2]$ is invariant for all $w\in( b_1,c_1)$, since 
\begin{equation}\label{eqn:eq1}
    [w,b_2]=[w,c_1]\cup (C\cap B).
\end{equation}

This, together with $f(B)=B$, implies $f(b_1)=b_1$.
Therefore, there exists $n_0\leq 0$ such that $b_1\not\in A_n$ for every $n\leq n_0$ since otherwise we would have that $b_1\in A_n$ for every $n\leq 0$, contradicting  $C\in\alpha(\{A_n\}_{n\leq 0})$.

{By~\eqref{eqn:eq1} the set $[w,b_2]\cup C$ is also invariant for any $w\in (b_1,c_1)$ and therefore for infinitely many $n$, thus for all, we must have $A_n\setminus (B\cup C)\neq \emptyset$. But $b_1\not\in A_n$ for almost all $n$, therefore $c_2\in \Int A_n$ for all $n<n_0$ (we decrease $n_0$ when necessary).
In particular $b_2\in \Int A_n$ for all $n\leq n_0$. }

Let $n_1<n_2<n_3\leq n_0$ such that $\dist_H(A_{n_i},C)<\eps/2$ for $i=1,3$ and $\dist_H(A_{n_2},B)<\eps/2$. 
Then $L=\overline{A_{n_1}\setminus B}$ is a subarc of $A_{n_1}$ with endpoint $b_2$.
By the choice of $n_i$, $0\leq i\leq 3$, we know that $b_1\notin f(L)$. 
But now $f(L)\supset L$ since otherwise $f^2(A_{n_1})\subset f(A_{n_1})$ by~\eqref{eqn:eq1} and consequently $A_{n_3}\subset A_{n_2}$, which contradicts the choice of $n_i$, $i=2,3$.
We proved $f(L)\supset L$. But the property of $A_{n_1}\cap B$ being invariant means that $\dist(b_1,f^{n_2-n_1}(L))<\eps/2$.
Hence $\dist(b_1,f^{n_3-n_1}(L))<\eps/2$ because $f^{n_3-n_1}(L)\supset f^{n_2-n_1}(L)$, implying $\dist_H(A_{n_3},C)>\eps/2$. A contradiction.

\medskip\paragraph{\textbf{Case 2:} $C\subset \Int B$}   
{Since 
 $B\in\alpha(\{A_n\}_{n\leq 0})$
and $C\subset \Int B$, we have that $C\subset A_n$ for infinitely many $n\leq 0$
 But $f(C)=C$ and hence $C\subset A_n$ for all $n\leq 0$.}

There are some $n_1<n_2<0$ such that $\dist_H(A_{n_1},C)<\eps/3$ and $\dist_H(A_{n_2},B)<\eps/3$.
Denote $D=A_{n_1}$ and $r=n_2-n_1>0$.
Now $D\subset \Int B$ and $D\subset f^r(D)$.
This implies that, for all $0\leq i< r$, $f^i(D)\subset f^r(f^i(D))$ and hence, for all $k\geq 0$, $f^k(D)\supset f^{i_k}(D)$ for some $0\leq i_k<r$.
Moreover, since $C\subset D$, $A_{n_2}=f^r(D)$ and $f(C)=C$, we have that $\dist_H(C, f^i(D))>0$ for all $0\leq i<r$, in particular
\[\delta=\min\{\dist_H(f^i(D),C)\colon 0\leq i<r\}>0.\]

Finally we conclude that $C\subset f^n(D)$ and $\dist_H(C, f^n(D))\geq \delta$ for all $n\geq 0$.
But $D\subset \Int B$, so $D\subset A_n$ for infinitely many $n\geq 0$, which implies that $\dist_H(C, A_n)\geq \delta$ for all $n\geq 0$, which is in contradiction with $C\in\alpha(\{A_n\}_{n\leq 0})$.

 Indeed, we have just shown that such $C$ cannot exist in either of the two cases, and this proves the lemma.
\end{proof}

\begin{lemma}\label{lem:(non)deg1}
    Let $A\in\mathcal{C}(I)$ and let $\{A_{n}\}_{n\leq 0}$ be its backward branch. If $\alpha(\{A_{n}\}_{n\leq 0})$ contains a nondegenerate $\tilde{f}$-periodic subinterval $B$ and $\{y\}\in\alpha(\{A_{n}\}_{n\leq 0})$ for some $y\in I$, then $\Per(f)\cap \Int B=\emptyset$.
\end{lemma}

\begin{proof}
    Suppose on the contrary, that there is some $z\in \Per(f)\cap \Int B$.
    From $B\in \alpha(\{A_{n}\}_{n\leq 0})$, it follows that $A_n\cap\mathcal{O}_f(z)\neq\emptyset$ for all $n\leq 0$.
    Therefore, $y\in \mathcal{O}(z)$ but then $\{y\}$ is periodic, contradicting Lemma~\ref{lem:periodic_alpha}.
\end{proof}

\begin{lemma}\label{lem:nondeg}
    Let $A\in\mathcal{C}(I)$ and let $\{A_{n}\}_{n\leq 0}$ be its backward branch. If $\alpha(\{A_{n}\}_{n\leq 0})$ contains a nondegenerate $\tilde{f}$-periodic subinterval $B$ then all the elements of $\alpha(\{A_{n}\}_{n\leq 0})$ are nondegenerate.
\end{lemma}

\begin{proof}
    Suppose that there is some $\{y\}\in \alpha(\{A_{n}\}_{n\leq 0})$.
    {As before, for simplicity, we may assume that $B$ is a fixed point of $\tilde{f}$, which is possible, because $\{f^i(y)\}\in \alpha(\{A_{n}\}_{n\leq 0})$, so
    $\alpha_{\tilde{f}^p}(\{A_{pn}\}_{n\leq 0})$ contains a singleton for every $p$.}
    Since $y\in s\alpha(x)$ for every $x\in A$, $y$ is either a periodic point, an element of a maximal solenoid $Q_{\max}$ or a nonperiodic element of a basic set by Theorem~\ref{thm:maximal} and Blokh's Decomposition theorem.

Like before, we easily see that $A_n\cap B\neq\emptyset$ for all $n\leq 0$ so, for every degenerate $\{z\}\in\alpha(\{A_{n}\}_{n\leq 0})$, $z\in B$.   
In particular, $y\in B$.

Note that $y$ cannot be periodic by Lemma~\ref{lem:periodic_alpha}.
Suppose $y\in Q_{\max}$. Note that $\alpha(\{A_{n}\}_{n\leq 0})$ is invariant and $\omega_f(y)=Q_{\min}$, so, for all $z\in Q_{\min}$, $\{z\}\in \alpha(\{A_{n}\}_{n\leq 0})$.
Therefore, $Q_{\min}\subset B$ and $Q_{\min}\subset\overline{\Per(f)}$ so $\Per(f)\cap B$ is infinite, a contradiction with Lemma~\ref{lem:(non)deg1}.

Now suppose $y$ is a nonperiodic point of a basic set. If it is eventually periodic then $f^m(y)$ is a periodic point for some $m\geq 0$ and $\{f^m(y)\}\in\alpha(\{A_{n}\}_{n\leq 0})$, contradicting Lemma~\ref{lem:periodic_alpha}.
Therefore, $y$ has infinite orbit contained in $B$, but $\mathcal{O}_f(y)\subset\overline{\Per(f)}$ so $\Per(f)\cap B$ is infinite, again contradicting Lemma~\ref{lem:(non)deg1}.
Indeed all the elements of $\alpha\left(\{A_n\}_{n\leq 0}\right)$ are nondegenerate.
The proof is complete.
    \end{proof}

\begin{lemma}\label{lem:per_orbit}
    Let $\{A_n\}_{n\leq 0}$ be a backward branch of $A=A_0\in\mathcal{C}(I)$ and suppose that there is some $\tilde{f}$-periodic and nondegenerate $B\in \alpha(\{A_n\}_{n\leq 0})$.
Then $\alpha(\{A_n\}_{n\leq 0})=\mathcal{O}_{\tilde{f}}(B)$
\end{lemma}

\begin{proof}
    By Lemma~\ref{lem:nondeg}, $\alpha(\{A_n\}_{n\leq 0})$ consists of nondegenerate subintervals. 
    By Lemma~\ref{lem:asympt_per}, every $C\in\alpha(\{A_n\}_{n\leq 0})$ is either $\tilde{f}$-periodic or wandering.

    If there is some wandering $C\in \alpha(\{A_n\}_{n\leq 0})$, then there exists some $y\in I$ such that $\{y\}\in\omega_{\tilde{f}}(C)$.
    But $\alpha(\{A_n\}_{n\leq 0})$ is closed and invariant, so $\{y\}\in\alpha(\{A_n\}_{n\leq 0})$, contradicting Lemma~\ref{lem:nondeg}.
    Therefore, $\alpha(\{A_n\}_{n\leq 0})$ consists only of the $\tilde{f}$-periodic subintervals so, by Lemma~\ref{lem:periodic_alpha}, it is a periodic orbit.
     \end{proof}

Now we are ready to prove Theorem~\ref{thm:alpha_char}, which offers full description of the $\alpha$-limit sets of backward branches in the induced system by encapsulating all the previously obtained results.

\begin{proof}[Proof of Theorem~\ref{thm:alpha_char}]
 By Theorem~\ref{thm:JO2}, every element of $\alpha\left(\{A_n\}_{n\leq 0}\right)$ is either asymptotically periodic or wandering.
    First suppose that there is some nondegenerate and asymptotically periodic element of $\alpha\left(\{A_n\}_{n\leq 0}\right)$. Then Lemmas~\ref{lem:asympt_per} and~\ref{lem:per_orbit} finish the proof{, because (a) holds}.
    For the second case, suppose that there is a nondegenerate and wandering element of $\alpha\left(\{A_n\}_{n\leq 0}\right)$.
{Then  Theorem~\ref{thm:JO2} combined with Lemma~\ref{lem:per_orbit} 
shows that all elements of
$\alpha\left(\{A_n\}_{n\leq 0}\right)$ are either degenerate or wandering.
    Now Lemma~\ref{lem:wandering} 
    finish the proof, showing that the case (c) holds.}
    The remaining case is when $\alpha\left(\{A_n\}_{n\leq 0}\right)$ consists of singletons.
    Take any $x_0\in A$ and its backward branch $\{x_n\}_{n\leq 0}$ such that $x_n\in {A_n}$ for all $n\leq 0$. 
    Denote $\alpha=\alpha\left(\{x_n\}_{n\leq 0}\right)$.
    Now, $\{y\}\in\alpha\left(\{A_n\}_{n\leq 0}\right)$ implies $y\in \alpha$ and vice-versa, i.e. (b) holds.
\end{proof}
    
\section{Special \texorpdfstring{$\alpha$}{alpha}-limit sets}\label{sec:7}

In this section, we study special $\alpha$-limit sets in the induced system $(\mathcal{C}(I), \tilde{f})$.
We begin with  examples that help build intuition about the $\alpha$-limit sets of different backward branches of a nondegenerate continuum.
We then prove Theorem~\ref{thm:pos_ent}, which provides a sufficient condition for the positivity of the topological entropy of the base system.
Finally, we show that $s\alpha(A)$ is closed whenever $A \in \mathcal{C}(I)$ is nondegenerate.

\begin{example}
    Let $f\colon[-3,3]\to[-3,3]$, \[
    f(x)=\begin{cases}
        2x+3, & x\leq -5/2,\\
        -2, & -5/2<x<-3/2,\\
        2x+1 & -3/2\leq x\leq-1,\\
        x, & -1<x<1,\\
        2x-1 & 1\leq x\leq3/2,\\
        2, & 3/2<x<5/2,\\
        2x-3, & x\geq 5/2.\\
    \end{cases}
    \]
Note that $\htop(f)=0$ since every periodic point is a fixed point.    
We also have \[s\alpha\left([-2,2]\right)=\{[-2,2],[-1,1],[-3,3],[-1,2],[-2,1],[-3,2],[-2,3],[-1,3],[-3,1]\}.\]
\end{example}

\begin{figure}[H]
\centering
\includegraphics[scale=0.6]{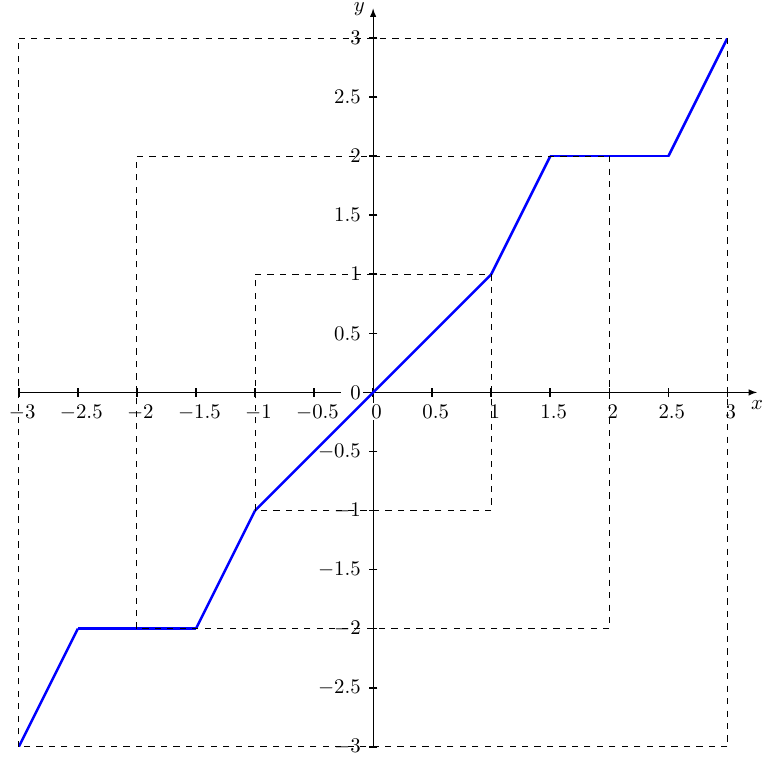}
\caption{}\label{fig:ex3}
\end{figure}

Note that in the above example, any two invariant subcontinua belonging to the same special $\alpha$-limit sets intersect with nonempty interior. Next example shows that it is not necessarily the case.

\begin{example}
    Let $f\colon[-1,1]\to[-1,1]$, \[
    f(x)=\begin{cases}
        2x+1, & x\leq -1/2,\\
        0, & -1/2<x<1/2,\\
        2x-1, & x\geq 1/2.\\
    \end{cases}
    \]
Note that $\htop(f)=0$ since $\Per(f)=\{-1,0,1\}$.    
Then \[s\alpha\left(\{0\}\right)=\{\{0\},\{-1\},\{1\},[-1,0],[0,1],[-1,1]\},\]
i.e. we have disjoint subcontinua lying in $s\alpha\left(\{0\}\right)$.
\end{example}

\begin{figure}[H]
\centering
\includegraphics[scale=1.2]{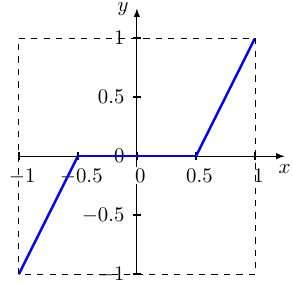}
\caption{}\label{fig:ex5}
\end{figure}

In previous examples any two nondegenerate elements of $s\alpha(A)$ were intersecting. 
At the same time, entropy of the maps was zero.
We will show it was not accidental. But first we need the following lemma.

\begin{lemma}\label{lem:positive_htop}
    Let $f\colon I\to I$ be an interval map, $A\in\mathcal{C}(I)$ and let there be nondegenerate $B,C\in\Fix(\tilde{f})\cap s\alpha(A)$, $B\cap C=\emptyset$ and let $d=\dist(B,C)>0$. Let $\{A_n\}_{n\leq 0}$ be a backward branch of $A$ such that $\{B\}= \alpha\left(\{A_n\}_{n\leq 0}\right)$. 
    Then there is some $t>0$ and an arc $L\subset A_{-2t}\cap N(B,d/2)$ such that $f^2(L)\supset L$ and $A\supset f^{2t}(L)\supset A\setminus N(B,d/2)$.
\end{lemma}

\begin{proof}
Note that, since $B$ is nondegenerate and invariant, we have that $A_k$ intersects $B$ for every $k\leq 0$, and, in particular $A=A_0$ intersects $B$.
By the same argument, $A$ also intersects $C$. {This in particular implies that $A_n\setminus B\neq \emptyset$ for any $n$.}

Denote $\eps=\min\{d/2,\diam B\}$. By the uniform continuity of $f$, there is ${\eps>}\delta>0$, such that $\dist(x,y)<\delta$ implies $\dist(f^i(x),f^i(y))<\eps$ for $i=1,2$. In particular $f^i(\N(B,\delta))\subset \N(B,\eps)$.
Take $k>0$ such that $A_{-k}\subset N(B,\delta)$. Since $B$ is invariant and $A\cap C\neq\emptyset$, there is an arc $D$, a component of $\overline{A_{-k}\setminus B}$ such that $f^k(D)\cap C\neq\emptyset$.
Note that $D$ and $B$ share a mutual endpoint. Denote $B=[b_1,b_2]$ and assume $b_1\in D$. By the choice of $\delta$, arcs $f(D)$ and $f^2(D)$ also contain exactly one endpoint of $B$, not necessarily at their boundaries.
We distinguish two cases, depending on whether $b_1\in f(D)$ or $b_2\in f(D)$:
\medskip\paragraph{\textbf{Case 1:} $b_1\in f(D)$.} Then $f(D)\supset D$, because otherwise, by the choice of $\delta$, $f(D)\subset D\cup B$ and therefore $f^k(D)\subset D\cup B \subset \N(B,\eps)$, a contradiction.

If there is another component $E$ of $\overline{A_{-k}\setminus B}$, i.e. the one containing $b_2$, such that $f^j(E)\setminus N(B,\eps)\neq\emptyset$ for some $j>0$ then we have two possibilities. Firstly, if $b_2\in f(E)$ then $f(E)\supset E$ by the similar arguments as above and we put $L=D\cup B\cup E=A_{-k}$. Secondly, if $b_1\in f(E)$ then, since $D\cap f(E)\neq\emptyset$ and by the choice of $\delta$, {either $b_1\in f^2(E)$ or $f^2(E)\subset B$}, {with the latter option excluded by the condition $f^j(E)\setminus N(B,\eps)\neq\emptyset$ for some $j>0$.}
This implies that $A_{-k+1}\setminus B\subset f(D)\cup f(E)$, $\overline{A_{-k+1}\setminus B}$  is an arc and $f(\overline{A_{-k+1}\setminus B})\supset \overline{A_{-k+1}\setminus B}$ {(by argument analogous to $f(D)\supset D$)}. In this case we put $L=\overline{A_{-k+1}\setminus B}$ and if there was not such arc $E$ we put $L=D$.

\medskip\paragraph{\textbf{Case 2:} $b_2\in f(D)$.} If $b_2\in f^2(D)$ then from the Case 1. we have that $f(F)\supset F$, where $F$ is a component of $\overline{A_{-k+1}\setminus B}$ containing $b_2$. Even more, $F=\overline{A_{-k+1}\setminus B}$ since $b_1\in D$ while $b_2\in f(D)\cap f^2(D)$, impliying that $f(E)\cap E\neq\emptyset$, where $E$ is a component of $\overline{A_{-k}\setminus B}$ containing $b_2$ and we can put $L=F$.
To complete the proof of this case,  it remains to consider when $b_2\in f(D)$ and $b_1\in f^2(D)$. By the same arguments as earlier, we have that $f^2(D)\supset D$,  because otherwise $f^{2}(D)\subset D\cup B$ but then $f^{k}(D)\subset D\cup f(D)\cup B$, a contradiction.
If there was another component $E$ of $\overline{A_{-k}\setminus B}$, i.e. the one containing $b_2$, such that $f^j(E)\setminus N(B,\eps)\neq\emptyset$ for some $j>0$ then necessarily $f^2(E)\supset E$ (observe that $b_2\in f^2(E)$ by construction) and we put $L=D\cup B\cup E$ and we fix $L=D$ when there is no such arc $E$.

In all the possible cases, there is some $j<0$ and $L\subset A_j$ such that $f^2(L)\supset L$ and $\cup_{n\geq 0}f^n(A_j\setminus L)\subset N(B,\eps)$. To prove the lemma, it just requires to ensure that $j$ is an even number. 
Obviously, if $j$ is odd, it is enough to replace it with $j+1$, while replacing $L$ with $f(L)$.
\end{proof}

\begin{proof}[Proof of Theorem~\ref{thm:pos_ent}]
It is sufficient to show that $f^n$ has a horseshoe for some $n\geq 0$.
Take $p\in\mathbb{N}$ such that $f^p(B)=B$ and $f^p(C)=C$.
Then there are some $0\leq i,j\leq p$ such that $f^i(B),f^j(C)\in s\alpha_{\tilde{f}^p}(A)$.
After replacing $f$ by $f^p$ and $B,C$ by their iterates, without loss of generality may assume that $p=1$ and $i=j=0$.

By Lemma~\ref{lem:per_orbit} there are backward branches $\{B_{k}\}_{k\leq 0}$ and $\{C_{k}\}_{k\leq 0}$ of $A$ such that $\lim\limits_{k\to-\infty}B_k =B$ and $\lim\limits_{k\to-\infty}C_k=C$.
    Let $d=\dist(B,C)>0$.
    By Lemma~\ref{lem:positive_htop}, there are $t_1,t_2>0$ and arcs $L_1\subset B_{-2t_1}\cap N(B,d/2)$, $L_2\subset C_{-2t_2}\cap N(C,d/2)$ such that $f^2(L_i)\supset L_i$ for $i=1,2$ and $A\supset f^{2t_1}(L_1)\supset A\setminus N(B,d/2)$, $A\supset f^{2t_2}(L_2)\supset A\setminus N(C,d/2)$.
By definition $A\setminus N(C,d/2)\supset L_1$ and $A\setminus N(B,d/2)\supset L_2$, yielding $f^{2t_1}(L_1)\supset L_2$ and $f^{2t_2}(L_2)\supset L_1$.
Now, for $t=\max\{t_1,t_2\}$ we have $f^{2t}(L_i)\supset L_1\cup L_2$ for $i=1,2$ and $L_1\cap L_2=\emptyset$, i.e. $f^{2t}$ has a horseshoe. The proof is completed.
\end{proof}

\begin{remark}\label{rem:horseshoe}
Theorem~\ref{thm:pos_ent} states that there are arcs $L_1$ and $L_2$ forming a horseshoe for some iterate of $f$.
    Note that this statement can be made stronger, that is for any given $\zeta>0$, we can ensure that $L_1\subset N(B,\zeta)$ and $L_2\subset N(C,\zeta)$.
    Indeed, in the proof of Lemma~\ref{lem:positive_htop}, $L_1\subset \bigcup_{i=0}^2 f^i(N(B,\delta))$ where $\delta$ can be chosen arbitrarily small (and
    $L_2$ is chosen analogously).
    We will use this fact later, in the proof of Lemma~\ref{lem:disjoint_seq}.
\end{remark}

\begin{example}
    Let $f\colon[-1,2]\to[-1,2]$, \[
    f(x)=\begin{cases}
    x, & x\leq 0,\\
        4x, &0< x< 1/2,\\
        -x+5/2, &1/2\leq x\leq 3/2,\\
        -4x+7, & x>3/2.\\
    \end{cases}
    \]
{Note that this map satisfies the assumptions of} Theorem~\ref{thm:pos_ent}.
Namely, $B=[-1,0]$ and $C=[1,3/2]$ are intervals {invariant for} $f$, both included in $s\alpha\left([-1,2]\right)$.
Indeed, $\{[-1,2^{-2n+1}]\}_{n\geq 0}$ is a backward branch of $[-1,2]$ such that $B=\lim_{n\to\infty}[-1,2^{-2n+1}]$.
Furthermore, $C=\lim_{n\to\infty}[a_{-n},b_{-n}]$ where $[a_{-2n},b_{-2n}]=[-2^{-2n+1}+1,2^{-2n-1}+3/2]$, for $n\geq 0$, in particular $[a_0,b_0]=[-1,2]$.
\end{example}    

\begin{figure}[h]
\centering
\includegraphics[scale=1]{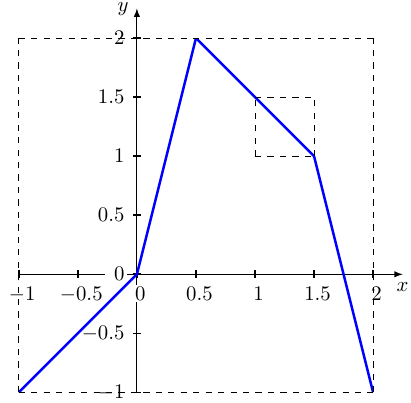}
\caption{}\label{fig:ex7}
\end{figure}

Now we focus on describing the situation in which different elements of the special $\alpha$-limit set of a nondegenerate $A \in \mathcal{C}(I)$ form a nested sequence.
We begin with an example.

\begin{example}
    Let $f\colon[-2,3]\to[-2,3]$, \[
    f(x)=\begin{cases}
        8x+14, & x\leq -3/2,\\
        -6x-7, & -3/2<x<-1,\\
        x & -1\leq x\leq 0,\\
        2x, & 0<x<0.5,\\
        3-4x & 0.5\leq x\leq1,\\
        -1, & 1<x<2,\\
        4x-9, & x\geq 2.\\
    \end{cases}
    \]

There is an uncountable family of fixed subintervals in a single special $\alpha$-limit set. Namely, $[a,0]\in s\alpha\left([-1,1]\right)$, for every $a\in[-1,0)$ (Figure~\ref{fig:ex8}).
\end{example}

\begin{figure}[H]
\centering
\includegraphics[scale=0.75]{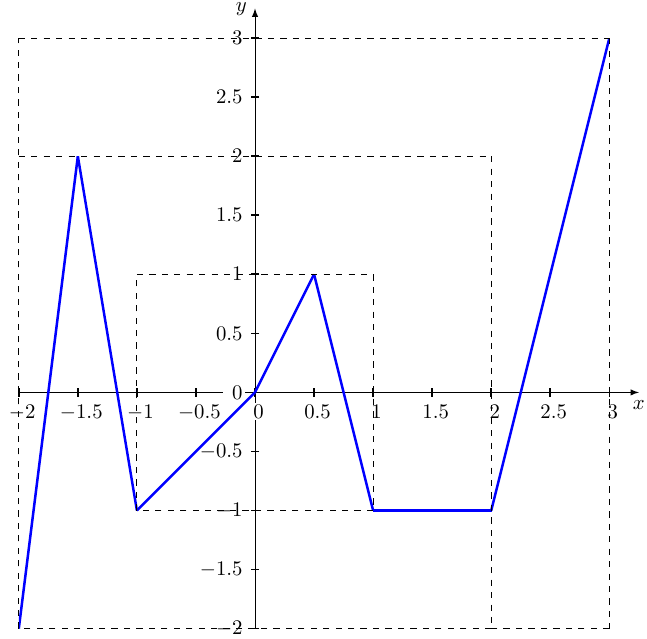}
\caption{}\label{fig:ex8}
\end{figure}
Note that in the above example all subintervals of the uncountable set \linebreak $\{[a,0],a \in[-1,0)\}$ share the common endpoint $0$.
On the other hand, we also have $[-1,1],[-2,2],[-2,3]\in s\alpha\left([-1,1]\right)$.
This situation is, in a sense, maximal, as shown by the following proposition.

\begin{proposition}\label{prop:nested}
     Let $f\colon I\to I$ be an interval map, $A\in\mathcal{C}(I)$ and let $\mathcal{B}$ be a family of nested subintervals such that $\mathcal{B}\subset s\alpha(A)$. 
     If there are $B_1,B_2,B_3\in\mathcal{B}$ such that $B_1\subset\Int B_2\subset \Int B_3$,  then $B_2=A$ and every $B\in\mathcal{B}$ shares an endpoint with either $B_1$ or $B_3$. Moreover, at most one $B\in\mathcal{B}\setminus \{B_3\}$ shares an endpoint with $B_3$.
\end{proposition}
\begin{proof}
    We begin by proving that $B_2=A$.  By Lemma~\ref{lem:per_orbit}, there are backward branches $\{B^i_n\}_{n\leq 0}$, $i=1,3$ of $A$, $\lim_{n\to -\infty}B^i_n=B_i$ for 
    $i=1,3$. 
    We have $B_1\subset\Int B_2\subset \Int B_3$, so $B^1_n\subset B_2$ and $B^3_n\supset B_2$ for all but at most finitely many $n\leq 0$. 
    {But $B_2$ is invariant, therefore }
    $B_2\subset B^3_0=A=B^1_0\subset B_2$ and the claim is proved.
    
    Now suppose there was some $B_4\in\mathcal{B}$ different from $B_i$, $i=1,2,3$, not sharing endpoint with any $B_i$, $i=1,3$.
    Then there are pairwise distinct $i,j,k\in\{1,2,3,4\}$ such that $B_i\subset \Int B_j\subset \Int B_k$ and $j\neq 2$. But then $B_j=A=B_2$, a contradiction.

    Now let us prove the last assertion of the Proposition. Assume that there are some $C,D,E\in\mathcal{B}$, such that $B_2\subset \Int C\subset D\subset E$ and that $C,D,E$ all have mutual endpoint $z$ {(which as we proved earlier is an endpoint of $B_3$)}.
    Let $B_2=[a,b]$.
    Furthermore, without loss of generality, let $C=[z,c]$, $D=[z,d]$ and $E=[z,e]$ for $z<a<b<c<d<e$.
    We will prove that $z$ and $d$ are fixed points of $f$. 
    If there was a periodic point $y\in(z,c]\setminus B_2$ then $\mathcal{O}_f(y)\cap B_2=\emptyset$ since $B_2$ is invariant.
    But $B_2=A$, so $D\in s\alpha(B_2)$ and $y\in\Int D$, hence $\mathcal{O}_f(y)\cap B_2\neq\emptyset$, which is a contradiction.
    Indeed, $\Per(f)\cap (z,c]\setminus B_2=\emptyset$. 
    Now Lemma~\ref{lem:fixed_sub_fixed} implies $f(z)=z$.
    Similarly, $(c,d)\cap\Per(f)=\emptyset$.
    Indeed, if there was some $y\in \Per(f)\cap (c,d)$ then $\mathcal{O}_f(y)\cap C=\emptyset$. But $y\in \Int D$ and $D\in s\alpha(B_2)$, again yielding $\mathcal{O}_f(y)\cap B_2\neq\emptyset$.  
    Now again by Lemma~\ref{lem:fixed_sub_fixed} we get $f(d)=d$.
    But this contradicts the existence of $E$ since $E\in s\alpha(B_2)$ and $d\in\Int E$ imply $d\in B_2$, a contradiction.
    \end{proof}

We end the section with a series of auxiliary results, which will be then used to prove Theorem~\ref{thm:special_closed}, stating that $s\alpha(A)$ is closed whenever $A\in\mathcal{C}(I)$ is nondegenerate.

\begin{lemma}\label{lem:no_overlap}
    Let $f\colon I\to I$ be an interval map and $A\in \mathcal{C}(I)$.
    Then there is no sequence $\{B_n\}_{n\in\mathbb{N}}$, $B_n=[a_n,b_n]$ in $s\alpha(A)\cap\Fix(\tilde{f})$ such that $a_n<a_{n+1}$, $b_n<b_{n+1}$ and $a_n<b_m$ for every $n,m\in\mathbb{N}$. 
\end{lemma}
\begin{proof}
    Note that $b_1\in \Int B_n$ for every $n>1$, so the elements of the sequence $\{B_n\}_{n\in\mathbb{N}}$ have pairwise nondegenerate intersection.

   First we claim that $B_n\cup B_m\in s\alpha(A)$ for every $1<n<m$.
   To prove this, let $\{A_k^n\}_{k\leq 0}$ and $\{A_k^m\}_{k\leq 0}$ be backward branches of $A$ converging to $B_n$ and $B_m$ respectively.
   Note that $B_1\cap B_{m+1}$ is invariant and $B_1\cap B_{m+1}\subset \Int B_n$ and $B_1\cap B_{m+1}\subset \Int B_m$.
   {Therefore }
   $B_1\cap B_{m+1}\subset A_k^n\cap A_k^m$ for every $k\leq 0$.
   Denote $A_k=A_k^n\cup A_k^m$ for every $k\leq 0$ and observe that $A_k$ is connected.
   Then, for every $k<0$, $f(A_k)=f(A_k^n\cup A_k^m)=f(A_k^m)\cup f(A_k^n)=A_{k+1}^{n}\cup A_{k+1}^m=A_{k+1}$ and $A_0=A_0^n\cup A_0^m=A$.
   In other words, $\{A_k\}_{k\leq 0}$ is a backward branch of $A$ converging to $B_n\cup B_m$ and the claim is proved.

   Now \[\{B_2\cup B_8,\ B_3\cup B_7,\ B_4\cup B_6,\ B_5\}\subset s\alpha(A)\] and \[B_5\subset \Int(B_4\cup B_6)\subset \Int(B_3\cup B_7)\subset \Int(B_2\cup B_8),\] contradicting Proposition~\ref{prop:nested}.
\end{proof}

\begin{lemma}\label{lem:increasing_seq}
    Let $f\colon I\to I$ be an interval map and $A\in \mathcal{C}(I)$.
    Assume that there is a 
    sequence $\{B_n\}_{n\in\mathbb{N}}$, $B_n=[a,b_n]$ in $s\alpha(A)\cap\Fix(\tilde{f})$ such that $a<b_n<b_{n+1}$ for every $n\in\mathbb{N}$ and let $B=\lim_n B_n$.
    Then also $B\in s\alpha(A)\cap\Fix(\tilde{f})$.   
\end{lemma}
\begin{proof}
    Denote $B=[a,b]$, where $b=\lim_{n\to\infty}b_n$.
    Let, for each $n\in\mathbb{N}$, $\{A^n_k\}_{k\leq 0}$ be a backward branch of $A$ converging towards $B_n$.
    Since $B_n$ are fixed points of $\tilde{f}$, we obviously have $A^n_k\cap B_n\neq\emptyset$ for every $n\in\mathbb{N}$ and every $k\leq 0$.

    We claim that $b\in A$.
    By Lemma~\ref{lem:fixed_sub_fixed}, $P_n=(b_n,b_{n+1}]\cap\Per(f)\neq\emptyset$, for every $n\in\mathbb{N}.$
    But $B_n$ and $B_{n+1}$ are invariant and therefore $P_n$ is invariant for every $n\in\mathbb{N}$.
    Now, for every $n\in\mathbb{N}$, $P_n\subset A_k^{n+2}$ for all $k\leq 0$ since $B_{n+2}=\lim_{k\to-\infty}A^{n+2}_k$.
    But then $P_n\subset A$ for every $n\in\mathbb{N}$. This proves the claim.

We need to check that $a\neq 0$. Indeed, $a=0$, together with $\lim_{k\to-\infty}A^1_k=B_1$, would imply $A_k^1\subset B_2$ for every $k\leq 0$ and, in particular, $A\subset B_2$. But $b\in A$, a contradiction. Therefore, $a>0$.

Now we claim that there is some $\delta>0$ such that, for every $0<\eps<\delta$, $f([a-\eps,a])\supsetneq [a-\eps,a]$.
Take $\delta>0$ such that  $\dist(x_1,x_2)<\delta$ implies $\dist(f(x_1),f(x_2))<b_2-b_1$.
Now, take any $0<\eps<\delta$ and let us prove that $f([a-\eps,a])\supsetneq [a-\eps,a]$.
Suppose on the contrary, that $f([a-\eps,a])\not\supsetneq [a-\eps,a]$. {This in particular implies that $f(x)\geq a-\eps$ for every $x\in [a-\eps,a]$.} 
Since $B_1$ is invariant, $f([a-\eps,a])$ is an arc intersecting $B_1$.
But now and by the choice of $\delta$, $f([a-\eps,a])\subset [a-\eps,b_2]$.
Therefore, $[a-\eps,b_2]$ is invariant and $A^1_k\subset[a-\eps,b_2]$ for every $k\leq 0$.
In particular, $A\subset [a-\eps,b_2]$. But $b\in A$, a contradiction.
We conclude that $f([a-\eps,a])\supsetneq [a-\eps,a]$ for every $0<\eps<\delta$, i.e. the claim is proved.
Note that this also implies that $f(a)=a$ since $B_1$ is invariant.
{Simply, if it was not the case, then $f(a)>a$ and so for sufficiently large $k$ we have $f(A_k^1)\subset [a,b_2])$ which is a contradiction as explained before.}

Next observe that for all $k\leq 0$ such that $\dist(A_k^1,B_1)<b_2-b_1$ we have $a\in A_k^1$. Otherwise $A_k^1\subset B_2$ and consequently $A\subset B_2$, contradicting $b\in A$. But $f(a)=a$ so $a\in A_k^1$ for every $k\leq 0$ and, in particular, $a\in A$.

We have already proved that $B\subset A$. But $B$  is a limit of a sequence in $\Fix(\tilde{f})$, hence $f(B)=B$.
It remains to show that $B\in s\alpha(A)$.

For every $k\leq 0$ denote $A_k=A_k^1\cup B\in\mathcal{C}(I)$
and observe that:
\[f(A_k)=f(A^1_k\cup B)=f(A_{k}^1)\cup f(B)=A_{k+1}^1\cup B= A_{k+1}\text{\ and}\] 
and \[A_0=A_0^1\cup B=A\cup B=A.\]
Furthermore, \[\lim_{k\to -\infty}A_k=\lim_{k\to -\infty}\left(A^1_k\cup B\right)=B_1\cup B=B.\]
This shows that indeed $B\in s\alpha(A)$ completing the proof.
\end{proof}
The following standard Lemma is a direct consequence of Darboux property.
\begin{lemma}\label{lem:sharkovskii}
    Let $f\colon I\to I$ be a continuous interval map.
    If $J,L$ are closed subintervals of $I$ such that $f(L)\supset J$ then there exists a closed subinterval $K$ of $L$ such that $f(K)=J$.
\end{lemma}

\begin{lemma}\label{lem:decreasing}
    Let $f\colon I\to I$ be an interval map and $A\in \mathcal{C}(I)$.
    Assume there is a converging sequence $\{B_n\}_{n\in\mathbb{N}}$, $B_n=[a,b_n]$ in $s\alpha(A)\cap\Fix(\tilde{f})$ such that $a<b_{n+1}<b_n$ for every $n\in\mathbb{N}$ and let $B=\lim_n B_n$.
    Then also $B\in s\alpha(A)$.   
\end{lemma}

\begin{proof}
{Denote $B=[a,b]$, where $b=\lim_{n\to\infty}b_n$. It may happen that $a=b$.}
    Let, for each $n\in\mathbb{N}$, $\{A^n_k\}_{k\leq 0}$ be a backward branch of $A$ converging towards $B_n$.
    Since $B_n$ are fixed by $\tilde{f}$, we obviously have $A^n_k\cap B_n\neq\emptyset$ for every $n\in\mathbb{N}$ and every $k\leq 0$.
    In particular, $A\cap B_n\neq\emptyset,$ for every $n\in\mathbb{N}$.

    Just like in the proof of Lemma~\ref{lem:increasing_seq}, $P=(b_3,b_2]\cap\Per(f)$ is a nonempty and strongly invariant set.
    But $B_1=\lim_{k\to-\infty} A_k^1$ and therefore $P\subset A_k^1$ for every $k\leq 0$ and, in particular, $b_3\in A$.

    Again, analogously to the proof of Lemma~\ref{lem:increasing_seq}, we first see that $a\neq 0$ and then claim that there is some $\delta>0$ such that, for every $0<\eps<\delta$, $f([a-\eps,a])\supsetneq [a-\eps,a]$.
If $B$ is nondegenerate, take $\delta>0$ such that  $\dist(x_1,x_2)<\delta$ implies $\dist(f(x_1),f(x_2))<\min\{b_4-b_5,1/2(b-a)\}$.
Otherwise, i.e. if $B=\{a\}$, we take $\delta>0$ such that  $\dist(x_1,x_2)<\delta$ implies $\dist(f(x_1),f(x_2))<b_4-b_5$.
Now, take any $\eps$, $0<\eps<\delta$ and let us prove that $f([a-\eps,a])\supsetneq [a-\eps,a]$.
Suppose on the contrary, that $f([a-\eps,a])\not\supsetneq [a-\eps,a]$, {which as before implies that $f(x)\geq a-\eps$ for every $x\in [a-\eps,a]$}. 
Since $B_5$ is invariant, $f([a-\eps,a])$ is an arc intersecting $B_5$.
But then by the choice of $\delta$, $f([a-\eps,a])\subset [a-\eps,b_4]$.
Therefore, $[a-\eps,b_4]$ is invariant and $A^5_k\subset[a-\eps,b_4]$ for every $k\leq 0$.
In particular, $A\subset [a-\eps,b_4]$. But $b_3\in A$, a contradiction.
We conclude that $f([a-\eps,a])\supsetneq [a-\eps,a]$ for every $0<\eps<\delta$, i.e. the claim is proved.
{Since $B_5$ is invariant, we can repeat argument form the proof of Lemma~\ref{lem:increasing_seq}, obtaining $f(a)=a$}.

We claim that for every $0<\eps<\delta$, there is some $z\in (a-\eps,a)$, such that $f([z,b])\supset [a-\eps,b]$ and $f([z,b])\cap[0,a-\eps)=\emptyset$.
Take any $\eps$, $0<\eps<\delta$ and note that $f([a-\eps,a])\supsetneq [a-\eps,a]$ implies that $f^{-1}(\{a-\eps\})\cap(a-\eps,a]\neq\emptyset$. Furthermore, $f^{-1}(\{a-\eps\})$ is a closed subset of $I$ so there is $z=\max \left(f^{-1}(\{a-\eps\})\cap[a-\eps,a]\right)$, $a-\eps<z<a$.
Note that $f([z,b])$ is a compact interval such that $[a-\eps,b]\subset f([z,b])$. 
By the choice of $\delta$, $f([z,b])\cap(b,1]=\emptyset$ if $a\neq b$.
Moreover, if $f([z,b])\cap[0,a-\eps)\neq\emptyset$, then $f^{-1}(\{a-\eps\})\cap(z,
{a}]\neq\emptyset$ 
contradicting the choice of $z$. The claim is proved. For the sake of convenience, for $\eps$ given, we will denote such $z$ by $Z(a-\eps)$. 

For any $
b_{(0)}\in(a-\delta,a)$ we can inductively define a sequence $\{b_{(n)}\}_{n\geq 0}$ as $b_{(n+1)}=Z(b_{(n)})$ for every $n\geq 0$.
From construction it follows that a sequence defined this way is strictly increasing and $a$ is its upper bound.
Also, we can prove that $\lim_ n b_{(n)}=a$. Suppose on the contrary, that $\lim_{n\to\infty}b_{(n)}=\Bar{a}<a$.
By definition $f([\Bar{a},a])\supsetneq [\Bar{a},a]$ so there is some $y\in (\Bar{a},a)$ and some $k\geq 0$ such that $f(y)=b_{(k)}$, which is a contradiction since $y>b_{(k+1)}$. 
Indeed, $\lim_{n\to\infty}b_{(n)}=a$.

Since $B_5=\lim_{k\to-\infty}A^5_k$, there is some $j<0$ such that $A^5_j\subset (a-\delta,b_4]$. 
Denote $A^5_j=[c,d]$. Remember that necessarily $c<a$.
But $B_4$ is invariant and $b_3\in A$ and $f(a)=a$, therefore $f^{-j}([c,a])=A$.
Take $b_{(0)}\coloneqq c$ and construct a sequence $\{b_{(n)}\}_{n\geq 0}$ like {described} above so that $f([b_{(n+1)},a])\supset[b_{n},a]$ and $b_{(n)}<b_{(n+1)}<a$ for every $n\geq 0$ and $\lim_n b_{(n)}=a$.

Now we are ready to finally construct a backward branch of $A$ converging towards $B$.
We divide the proof in two cases.
Let us first consider the case when $B=\{a\}$.
Now we inductively define a sequence of arcs. First set $J_0=[c,a]$ and next, if for some $n\geq 0$ the set $J_{n}\subset [b_{(n)},a]$ is already defined, we define $J_{n+1}$ to be any arc contained in $[b_{(n+1)},a]$ such that $f(J_{n+1})=J_n$.
Such an arc exists by Lemma~\ref{lem:sharkovskii} since $f([b_{(n+1)},a])\supset[b_{(n)},a]$.
Since $\lim_{n}b_{(n+1)}=a$, it follows that $\lim_n J_n=\{a\}$.
Now set $A_k=f^{k-j}([c,a])$ for all $k$ such that $j\leq k\leq 0$ and recall that by the definition of $[c,d]$ we have $A_0=A$.
For $k<j$ put $A_k=J_{j-k}.$
Therefore, $\{A_k\}_{k\leq 0}$ is a backward branch of $A$ converging to $B=\{a\}$, proving the Lemma in the  case of degenerate $B$.

Now suppose that $B$ is nondegenerate.
Let us first check that $B\subset A$. We have already proved that $b_3\in A$.
Suppose $a\notin A$. Since $f(a)=a$, we have $a\notin A^5_k$ for every $k\leq 0$.
But then $A^5_k\subset B_4$ for every $k\leq 0$ and, in particular, $A\subset B_4$. But $b_3\in A$ ,a contradiction.
Therefore, $B\subset[a,b_3]\subset A$.

For all $j\leq k\leq 0$
denote $A_k\coloneqq (f^{k-j}([c,a])\cup B)\in\mathcal{C}(I)$ .
For $k<j$ we define $A_k=[b_{(j-k)},a]\cup B$.
First of all, note that $A_0=A\cup B=A$.
{By the definition of $\delta$
we see that $f(b_{(k+1)},a)\subset f([a-\eps,a])\subset [0,b]$ and by the definition of $b_{(k+1)}$ we have $f(b_{(k+1)},a)\subset [b_{(k)},1]$.
This shows that $f(A_k)=A_{k+1}$ for every $k<0$.}
Finally, observe that $\lim_{k\to -\infty}A_k=\{a\}\cup B=B$.
Indeed, in both cases we have $B\in s\alpha(A)$ and the Lemma is proved.
\end{proof}

\begin{lemma}\label{lem:sequences}
    Let $\{B_n\}_{n>0}$ be a converging sequence of nondegenerate elements of $\mathcal{C}(I)$. Then either there exists a subsequence whose all elements intersect {(all of them contain at least one common point)} or there is a subsequence with pairwise disjoint elements. 
\end{lemma}
\begin{proof}
    If $B=\lim_n B_n$ is nondegenerate then there exists a subsequence of $\{B_n\}_{n>0}$ whose every element contains a midpoint of $B$, so there is nothing to prove.
    
    Next suppose $B=\{b\}$.
    If there are infinitely many elements of $\{B_n\}_{n>0}$ containing $b$ then we are done.
    Now suppose that $b\in B_n$ for finitely many $n>0$ and denote $B_n=[a_n,b_n]$.
    {Going to a subsequence when necessary, we may assume 
that both sequences $a_n,b_n$
    }
    are strictly monotone and either $a_n<b_n<b$ for all $n>0$
    or $b<a_n<b_n$ for all $n>0$.
    It is clear that $\diam B_{n}\to 0$ and $\lim a_{n}=
    \lim b_{n}=b$
 therefore,  for every $k>0$, $b_k\in B_n$ for at most finitely many $n$. Denote $n_1=1$ and then inductively put $n_{k+1}=\min \{j\in\mathbb{N}\colon j>n_k,\ b_{n_k}\notin B_j\}$. 
{By the construction intervals $\{B_{n_k}\}_{k>0}$ are pairwise disjoint.}
    \end{proof}

The proof of the following Lemma can be found in~\cite[Corollary 4.3]{Jelic2}.
\begin{lemma}
    Let $T$ be a tree and denote $m=\lcm\{1,2,\ldots,|\End(T)|\}$. Let $f\colon T\to T$ be a tree map and let $P_1$ and $P_2$ be two intersecting periodic points of $\left(\mathcal{C}(T),\tilde{f}\right)$ of periods $p_1$ and $p_2$, respectively.
    If $p_1>mp_2$ then $P_1\subset P_2$.
\end{lemma}
\begin{corollary}\label{cor:nested}
Let $f\colon I\to I$ be an interval map and let $P_1$ and $P_2$ be two intersecting periodic points of $\left(\mathcal{C}(I),\tilde{f}\right)$ of periods $p_1$ and $p_2$, respectively.
If $p_1>2p_2$ then $P_1\subset P_2$.
\end{corollary}

\begin{lemma}\label{lem:disjoint_seq}
    Let $f\colon I\to I$ be an interval map and $A\in\mathcal{C}(I)$ nondegenerate.
    Assume that a sequence $\{B_n\}_{n\in\mathbb{N}}\subset s\alpha(A)$ consists of pairwise disjoint elements and $B=\lim_n B_n$ exists. Then $B\in s\alpha(A)$.
\end{lemma}

\begin{proof}
    First note that $\lim_n(\diam B_n)=0$ and denote $B=\{b\}$.
    By Theorem~\ref{thm:alpha_char}, we can either assume that the sequence $\{B_n\}_{n\in\mathbb{N}}$ consists of wandering subintervals or that it consists of periodic subintervals.
    Moreover, without loss of generality, we may assume that $\{B_n\}_{n\in\mathbb{N}}$ is monotone in $I$, say  $a_n<b_n<a_{n+1}$ for every $n\in\mathbb{N}$ where $B_n=[a_n,b_n]$ (decreasing case is symmetric). 

    First, suppose that every $B_n$ is periodic with period $p_n$.
Let us clarify that Remark~\ref{rem:horseshoe} can be applied.
To that end, we consider two subcases. 
For the first subcase, suppose that the set $\{p_n\colon n\geq 0\}$ is bounded.
Let $p=\lcm\{p_n\colon n\in\mathbb{N}\}$. Obviously $f^p(B_n)=B_n$ for every $n\in\mathbb{N}$ and $f^p(B)=B$.
For every $n\in\mathbb{N}$. Let $\{A^n_k\}_{k\leq 0}$ be a backward branch of $A$ such that $B_n\in\alpha\left(\{A^n_k\}_{k\leq 0}\right)$.
Then, for every $n\in\mathbb{N}$ there is an integer $j_n\in [0,p)$ such that $\lim_{k\to -\infty} A^n_{kp+j_n}=B_n$.
There is a constant subsequence of $\{j_n\}_{n\geq 1}$ so we can assume that all $j_n$ are equal.
Now, for every $n\in\mathbb{N}$ we have $B_n\in s\alpha_{f^p}\left(f^{j_n}(A)\right)$.
Therefore, after replacing $f$ with $f^p$ and $A$ with $f^{j_n}(A)$, we can assume that $j_n=0$ for every $n\in\mathbb{N}$ and that $p=1$.
In other words, in this subcase, we let $\{B_n\}_{n\in\mathbb{N}}$ be a sequence of nondegenerate and pairwise disjoint fixed points of $f$ contained in $s\alpha(f)$ and converging to a subinterval $B$.
Such a sequence obviously satisfies the condition of Theorem~\ref{thm:pos_ent}, and consequently Remark~\ref{rem:horseshoe} can be applied.

Now we consider the alternative subcase,  that is $\{p_n\colon n\geq 1\}$ is unbounded and suppose that 
Theorem~\ref{thm:pos_ent} cannot be applied for any subsequence of $\{B_n\}_{n\geq 1}$.
This implies that all but finitely many elements of $\{B_n\}_{n\geq 1}$ intersect orbits of infinitely many elements of $\{B_n\}_{n\geq 1}$.
By Corollary~\ref{cor:nested}, there is a nested sequence $\{C_n\}_{n\geq 1}$ of periodic subintervals of $I$, every $C_n$ being an iterate of some $B_{k_{n}}$.
After passing to a subsequence, we can assume that $p_{k_{n+1}}/p_{k_n}\geq 3$ for every $n\geq 1$. But then we have a sequence $\{f^{i_n}(C_n)\}_{n>0}$ where, for every $n\in\mathbb{N}$, $0\leq i_n<p_n$ is chosen in a way that $f^{i_{n+1}}(C_{n+1})\subset \Int f^{i_n}(C_n)$. This yields a contradiction with Proposition~\ref{prop:nested}, since $s\alpha(A)$ is invariant.
Again we see that the sequence $\{B_n\}$ is such that the assumptions of Theorem~\ref{thm:pos_ent} are satisfied, and consequently Remark~\ref{rem:horseshoe} can be applied also in this case.

 Therefore, we may apply Remark~\ref{rem:horseshoe} after replacing  $\{B_n\}_{n\geq 1}$ with its subsequence if necessary, that is for every $n>1$ there is some $L_n\in \mathcal{C}(I)$ and $k_n\in\mathbb{N}$ such that $L_n\subset (b_{n-1},a_{n+1})$, $f^{k_n}(L_n)\cap B_1\neq\emptyset$ and  $f^{k_n}(L_n)\supset L_n$.
    We can now define a backward branch of $A$ converging towards $\{b\}$.   
    Since $B_2\in s\alpha(A)$, there is some $D\in \mathcal{C}(I)$, $D\subset (b_1,a_3)$, and some $j>0$ such that $f^j(D)=A$.
    For $-j\leq k\leq 0$, set $A_k=f^{j+k}(D)$.
    Next, the construction follows by the induction in the following way. Assume that for some $k\leq -j$ the set $A_k\subset (b_{m-1},a_{m+1})$ is already defined. Observe that $L_{m+2}\subset (b_{m+1},a_{m+2})$
    and therefore $f^{k_{m+2}}( L_{m+2})\supset [b_1,b_{m+1}]\supset A_k$.
     We apply Lemma~\ref{lem:sharkovskii} to find some closed interval $J_{m+2}\subset L_{m+2}$ such that $f^{k_{m+2}}(J_{m+2})=A_k$. 
     Now, for  $k-k_{m+2}\leq t\leq k$ we set $A_{t}=f^{t-k+k_{m+2}}(J_{m+2})$. We proceed by performing inductive step at $A_{k-k_{m+2}}$. It is clear that $B\in \alpha(\{A_n\}_{n\leq 0})$.

    Now suppose that every $B_n$ is wandering.
    By Lemma~\ref{lem:wandering}, for every $n\in\mathbb{N}$, there is some basic set $\omega_n=B(K^n,f)$ and a set $\mathcal{K}^n$ of intervals described in Remark~\ref{rem:model_BS}, such that $A$ is contained in some element of $\mathcal{K}^n$ and $\omega_{\tilde{f}}(A)=\mathcal{K}^n$.
    But then, {since $\mathcal{K}^n=\mathcal{K}^m$ for $m\neq n$, } there is a basic set $\omega=B(K,f)$ such that $\omega=\omega_n$, $K=K^n$ and $\mathcal{K}=\mathcal{K}^n$ for every $n\in\mathbb{N}$.
    Then we also have $b\in\omega$.
    Since all $B_n$ are pairwise disjoint and there are infinitely many of them, there is some $t\in\mathbb{N}$ and some $i\geq 0$ such that $B_t$ is included in the interior of some $K_{(i)}\in \mathcal{K}$.
    Then there is some $r>0$ and some $A_{-r}\in\mathcal{C}(I)$ such that $f^{r}(A_{-r})=A$ and $A_{-r}\subset\Int \mathcal{K}_{(i)}$.

    {For $-r\leq k\leq 0$ we put
    $A_k=f^{k+r}(A_{-r})$ and for remaining $A_k$ we proceed by induction in the following way.}
    There is some $m>0$ and a sequence of compact intervals $J_n=[j_n,b]$ such that $J_n\subset \mathcal{K}_{(m)}$ for every $n\in\mathbb{N}$ and $\lim_n j_n=b$.
    Note that $(\Int J_n)\cap\omega\neq\emptyset$ and therefore $\omega_{\tilde{f}}(J_n)=\mathcal{K}$ for every $n\in\mathbb{N}$.
    {It is clear that for every $n$ there is $s$ such that $f^s(J_n)\supset A_{-r}$ and if $J_k\subset \Int \mathcal{K}_{(m)}$ then $f^s(J_n)\supset J_k$. So the only possibly problematic situation is when $b$ is an endpoint of $\mathcal{K}_{(m)}$.
    If there is $y\in \Int \mathcal{K}_{(m)}$ such that $f^t(y)=b$ then we easily find $s>t$ such that $j_k\in f^s(J_n)$ and $y\in f^{s-t}(J_n)$ so $f^s(J_n)\supset J_k$.
    When such $y$ does not exist, then $b$ is a periodic point, and then $f^s(J_n)\supset J_k$ for sufficiently large $s$ which is multiple of the period of $b$.}
    Now we define a sequence $\{k_n\}_{n\in\mathbb{N}}$ of integers inductively.
    Set $k_1\coloneqq -r$.
    Assume that $k_n$ is defined
    and either $A_{-r}=A_{k_n}$ or $J_t\supset A_{k_n}$ for some $t$.
    Then there is $s_n>0$ such that $f^{s_n}(J_n)\supset A_{k_n}$.
    By Lemma~\ref{lem:sharkovskii} there is an interval $L_n\subset J_n$ such that $f^{s_n}(L_n)=A_{k_n}$.
    Now set $k_{n+1}\coloneqq k_n-s_n$, $A_{k_{n+1}}\coloneqq L_n$ and, for every $k_{n+1}<k<k_n$, set $A_{k}=f^{k-k_{n+1}}(A_{k_{n+1}})$.
    By the construction, obtained sequence $\{A_k\}_{k\leq 0}$ is a backward branch of $A$ and, since $A_{k_{n}}\subset J_n$, we have $\lim_{k\to-\infty}A_k=\{b\}$.
    Therefore, $\{b\}\in \alpha(\{A_n\}_{n\leq 0})\subset s\alpha(A)$. This completes the proof.
\end{proof}

Now we have all the tools to prove the main theorem of this section.

\begin{proof}[Proof of Theorem~\ref{thm:special_closed}]
    Let $A\in\mathcal{C}(I)$ be nondegenerate and $B\in \overline{s\alpha(A)}$.
Let $\{B_n\}_{n>0}$ be a sequence in $s\alpha(A)$ such that $\lim B_n=B$.
We divide the proof in two cases possible which follow from the statement of Lemma~\ref{lem:sequences} (we go to a subsequence of $\{B_n\}_{n>0}$ when necessary).

\medskip\paragraph{\textbf{Case 1:}} 
First assume that all the elements of $\{B_n\}_{n>0}$ intersect.
{Then we may also assume that all $B_n$ are nondegenerate, as otherwise $s\alpha(A)\ni B_n=B_m=B$ for every $m,n$ and so there is nothing to prove.
 By Theorem~\ref{thm:JO2}, and after taking subsequence if necessary, either every $B_n$ is asymptotically periodic or they are all wandering.
If they are wandering, then Theorem~\ref{thm:alpha_char}(c) applies, and so since $B_n\cap B_m\neq\emptyset$ for $m\neq n$, we must have $B_n=B_m=B$, which completes the proof. So the only possibility is Theorem~\ref{thm:alpha_char}(a), that is every $B_n$ is periodic with some period $p_n>0$.}
Notice that, by Proposition~\ref{prop:nested}, there cannot exist any set of indices $\{n_1,n_2,n_3,n_4\}\subset \mathbb{N}$ such that $B_{n_1}\subset\Int B_{n_2}\subset\Int B_{n_3}\subset\Int B_{n_4}$

We claim that the set $\{p_n\colon n>0\}$ of periods is bounded. To prove this we suppose on the contrary, that there is no upper bound for $\{p_n\colon n>0\}$.
{Going to a subsequence when necessary, by
Corollary~\ref{cor:nested}, 
we obtain that the sequence $B_n$
is nested, that is $B_{n+1}\subset B_n$
for every $n$. We may also assume that $p_{n+1}>3p_n$ for every $n$.
But then $p_{n}$ divides $p_{n+1}$ and $p_{n+1}/p_{n}\geq 3$, for every $n\in\mathbb{N}$.}
In particular, for every $n\in\mathbb{N}$, every iterate of $B_n$ contains at least three distinct iterates of $B_{n+1}$.
This way, just like in Lemma~\ref{lem:disjoint_seq}, we have a sequence $\{f^{i_n}(B_n)\}_{n>0}$ where, for every $n\in\mathbb{N}$, $0\leq i_n<p_n$ is chosen in a way that $f^{i_{n+1}}(B_{n+1})\subset \Int f^{i_n}(B_n)$. This yields a contradiction {with Proposition~\ref{prop:nested}, since $s\alpha(A)$ is invariant.}
Indeed, the set $\{p_n\colon n>0\}$ of periods is bounded.

Now, repeating exactly the same arguments as in the proof of Lemma~\ref{lem:disjoint_seq}, from now on we may assume that $\{B_n\}_{n\in\mathbb{N}}$ consists of nondegenerate and pairwise intersecting fixed points of $f$ contained in $s\alpha(f)$ and converging to a 
subinterval $B$. Now, if $\{B_n\}_{n\in\mathbb{N}}$ has a constant subsequence then obviously $B\in s\alpha(A)$. Otherwise, assumptions of one of the Lemmas~\ref{lem:no_overlap},~\ref{lem:increasing_seq}, and~\ref{lem:decreasing} are satisfied, proving that $B\in s\alpha(A)$. This completes the proof of first case.

\medskip\paragraph{\textbf{Case 2:}} 
Now consider the second case and assume that $\{B_n\}_{n\in\mathbb{N}}$ consists of pairwise disjoint intervals.
This case is covered by Lemma~\ref{lem:disjoint_seq}, which ensures $B\in s\alpha(A)$. 

We proved that $s\alpha(A)$ is closed, finishing the proof of 
Theorem~\ref{thm:special_closed}.
\end{proof}

Note that in~\cite{HantakovaRoth}, the authors proved that, for an interval map $f\colon I \to I$ and $x \in I$, the set $s\alpha(x)$ is not closed if and only if $x$ belongs to a maximal solenoidal set that contains a nonrecurrent point from the Birkhoff center of $f$.
Theorem~\ref{thm:special_closed} extends this result to the induced map, thus proving the following.

\begin{corollary}
    Let $f\colon I\to I$ be an interval map and $A\in \mathcal{C}(I)$. 
    Then $s\alpha(A)$ is not closed if and only if $A=\{a\}$ is a singleton and $a$ belongs to a maximal solenoidal set that contains a nonrecurrent point from the Birkhoff center of $f$.
\end{corollary}
\begin{proof}
{If $s\alpha(A)$ is not closed then by Theorem~\ref{thm:special_closed}
we have $A=\{a\}$ for some $a\in I$.
But then we can view $s\alpha(A)$
as not-closed $s\alpha(a)$ for $f$, and these are characterized in \cite{HantakovaRoth} exactly as in the statement. By exactly the same characterization in \cite{HantakovaRoth} we obtain that converse implication.
}
\end{proof}

\section*{Acknowledgements}

Piotr Oprocha is grateful
to University of Science and Technology in Hefei, China for its hospitality and financial
support.

Research of Domagoj Jeli\'c was supported by the European Union - NextGenerationEU, project: IP-UNIST-44 (ITPEM).

This article has been supported by EU funds under the project “Increasing the resilience of power grids in the context of decarbonisation, decentralisation and sustainable socioeconomic development", CZ.02.01.01/00/23\_021/0008759, through the Operational Programme Johannes Amos Comenius. 

\bibliographystyle{plain}
\bibliography{bibInducedAlphaLimit}	
\end{document}